\documentclass[11pt,a4paper]{article}
\usepackage{amsfonts,amsmath,amstext,amsbsy,amsthm,amscd}

\newcommand{\xr}{ \mathcal{X}}
\newcommand{\yr}{ \mathcal{Y}}
\newcommand{\Lr}{\mathcal{L}}
\newcommand{\psH}[2]{ {\left\langle #1, #2 \right\rangle}_{H} }
\newcommand{\pUV}{\mathfrak{p}}
\newcommand{\D}{\mathbb{D}}
\newcommand{\aug}{\text{aug}}
\newcommand{\obs}{\text{obs}}
\newcommand{\Xb}{\overline{X}}
\newcommand{\abs}[1]{\left\lvert#1\right\rvert}

\newcommand{\ind}[1]{{\mathbf 1}_{\left\{#1\right\}}}

\newcommand{\Block}{\mathbb{B}}

\newcommand{\norme}[1]{ {\left\lVert  #1\right\rVert}}
\newcommand{\normeDqp}[1]{ {\left\lVert  #1\right\rVert}_{q,p}}

\newcommand{\field}[1]{\mathbb{#1}}
\newcommand{\R}{\field{R}}
\newcommand{\PP}{\field{P}}

\newcommand{\F}{{\mathcal F}}

\newcommand{\W}{{\mathcal W}}

\newcommand{\dd}{\text{d}}

\newcommand{\ve}{\varepsilon}

\theoremstyle{plain}
\newtheorem{thm}{Theorem}

\newtheorem{cor}{Corollary}
\newtheorem{lem}{Lemma}
\newtheorem{prop}{Proposition}
\newtheorem{rem}{Remark}
\newtheorem{example}{Example}

\title{LAMN property for hidden processes:\\ the case of integrated diffusions}

\author{Arnaud GLOTER
\thanks{ Universit\'e de Marne--la--Vall\'ee, Laboratoire d'Analyse et de Math\'ematiques Appliqu\'ees UMR 8050,
5 Boulevard Descartes, 77454 Marne--la--Vall\'ee Cedex 2, FRANCE
 - email: arnaud.gloter@univ-mlv.fr - Corresponding author}  \ and
 Emmanuel GOBET \thanks{
ENSIMAG, INP Grenoble, Laboratoire Jean Kuntzmann  UMR 5224, B.P. 53, 38041 Grenoble Cedex 9, FRANCE - email: emmanuel.gobet@imag.fr}}

\begin{document}

\maketitle

\begin{abstract}
In this paper we prove the Local Asymptotic Mixed Normality (LAMN)
property for the statistical model given by the observation of local
means of a diffusion process $X$. Our data are given by $ \int_0^1
X_{\frac{s+i}{n}} \dd \mu (s)$ for $i=0,\dots,n-1$ and the unknown
parameter appears in the diffusion coefficient of the process $X$
only. Although the data are nor Markovian neither Gaussian we can
write down, with help of Malliavin calculus, an explicit expression
for the log-likelihood of the
model, and then study the asymptotic
expansion. We actually find that the asymptotic information of this
model is the same one as for a usual discrete sampling of $X$.
\vspace{0.5cm}\\
{\bf [Fran\c{c}ais]} Dans ce papier nous d\'emontrons la
propri\'et\'e LAMN pour le mod\`ele statistique constitu\'e par
l'observation des moyennes locales d'une diffusion $X$. Nos
donn\'ees sont d\'efinies comme $\int_0^1 X_{\frac{s+i}{n}} \dd \mu
(s)$ avec $i=0,\dots,n-1$ et le param\`etre inconnu appara\^it
seulement dans le coefficient de diffusion du processus $X$. Bien
que cette observation ne soit ni Gaussienne ni Markovienne nous
pouvons, par le calcul de Malliavin, obtenir une expression pour la
log-vraisemblance du mod\`ele. Nous sommes alors capables de
calculer l'information asymptotique  et montrons qu'elle est la
m\^eme que pour l'observation ponctuelle de la diffusion.\\[5mm]
{\em To appear in Annales de l'Institut Henri Poincare (B) Probability and Statistics.}
\end{abstract}

{{ K}{\footnotesize EYWORDS}:} Diffusion processes, parametric estimation,
LAMN property, Malliavin calculus, non-Markovian data
\noindent \vspace{0.2cm} \\
{AMS 2000 {{\footnotesize SUBJECT CLASSIFICATION:
} 60Fxx; 60Hxx; 62Fxx; 62Mxx}}

\section{Statement of the problem and main results}\label{S:Main}
\subsection{Introduction}

~\indent {\bf Model.} Let us consider the family of strong solutions $X^\theta$ to the following scalar equation
\begin{align}\label{E:EDS_theta}
  \dd X^\theta_t&=a(X^\theta_t,\theta) \dd B_t+b(X^\theta_t) \dd t,
  \\
\label{E:cond_init_xi} \quad X^\theta_0&=\xi_0,
\end{align}
where $B$ is a one-dimensional Brownian motion. We suppose that
$\theta$ lies in some compact interval $\Theta$ of $\mathbb{R}$ and
that $\xi_0$ is a real constant, which does not depend on $\theta$
and thus is known to the statistician.

{\bf Observations.} We consider $\mu$ some probability measure on
$[0,1]$ and assume
that our observation of the process is given by the local means of $X$
associated with this measure, with sampling of size $1/n$:
$$
\text{\bf (observations)} \qquad \Xb_j=\Xb_{j,n}:=\int_0^1 X_{\frac{s+j}{n}} \dd \mu (s), \text{ for }j=0,\dots,n-1.\qquad
$$
In the sequel this case is referred to as the integrated diffusion
case. This is an indirect observation of the process $X$ and the
observation is no more the realization of a Markov chain. Thus, this
framework is deeply related to the inference of hidden processes. We
assume that $\mu$ does not depend on $\theta$ and is known by the
statistician. When $\mu$ is equal to the Lebesgue measure, the
observation is the discrete sampling of $I_t=\int_0^t X_s ds$. This
is presumably the simplest case of the observation of only one
component of a bidimensional diffusion process $(X_t,I_t)_{0\le t\le
1}$, which is known in the literature as the standard integrated
diffusion case. Clearly, the usual case of pointwise observation of
$X$ is obtained if $\mu$ is some Dirac measure. However we will
exclude that the measure has mass only on the end points of the
interval and hence make the assumption:
\begin{equation}\label{E:hypomu}
\mu((0,1))>0.
\end{equation}
This paper is concerned with the Local Asymptotic Mixed Normality property of this statistical model.

{\bf Motivation.} Taking as the observation the integrated process
is actually quite natural. For instance, it arises when the
realization of the process has been observed after passage through
an electronic filter. Also, in random mechanics (see Kr\'ee and
Soize \cite{kree:soiz:86}), $X$ models the velocity of the system
and in general, we observe its position, i.e. the integral of $X$.
The modeling of ice-core data can be made through an integrated
diffusion process (see Ditlevsen, Ditlevsen and Andersen
\cite{ditl:ditl:ande:02}). Integrated processes also play an
important role in finance, when modeling the stochastic volatility
(see for instance Barndorff-Nielsen and Shephard \cite{barn:shep:01}
and references therein).

{\bf Literature background}. Despite of these numerous motivations,
few statistical studies deal with this situation.
Gloter \cite{glo1} \cite{glo2} provides an estimator in the
multiplicative case $a(x,\theta)=\theta a(x)$ and proves its
consistency and asymptotic normality.
The case of a low frequency observation
(local means over interval of length 1) is studied by Ditlevsen and
S\o rensen \cite{ditl:sore:04}, using prediction-based estimating
functions. On the other hand, for a direct observation of the
diffusion $X$, there are many contributions in the literature: see
Genon-Catalot and Jacod \cite{VGC}, Prakasa-Rao \cite{prak:99} and
references therein. None of these works deal with the problem of
optimal estimation in the integrated diffusion model.

Here, we directly address the problem of the LAMN property, whose fundamental
 consequence is to provide
 information on the minimal dispersion for an estimator of the parameter
 $\theta$ (see Ibragimov and Has'minskii \cite{Ibragimov}, Jeganathan \cite{Jeganathan1}
 \cite{Jeganathan2}, Prakasa-Rao \cite{prak:99}, Le Cam and Lo Yang \cite{leca:loya:00}).
 Such properties, for the observation of a discrete sampling of the diffusion,
 have been established in the one-dimensional
 setting by Dohnal \cite{dohn:87}, and then extended by Gobet \cite{Gobet1}
 \cite{Gobet2} to the multidimensional setting, both in the high frequency and
 ergodic framework. For this, Malliavin calculus techniques were used and paved
 the way to possibly handle more general situations than Markovian observations.
 This is exactly this way we follow in this work, to tackle the case of
 integrated diffusion.

{\bf Outlook.} We guess that this model captures the main
difficulty of most hidden models: the lack of Markov property for
the observation. Hence the method developed below
(augmented observation, Malliavin calculus representation, Gaussian approximation) may be useful to treat more general
situations. Among the natural situations coming from applications, one can think of the measurement of a
stochastic phenomenon blurred by some noise, or stochastic volatility models widely used in finance \cite{VCJL98}.
This can be formalized as follows: the system ${\cal X}^\theta$ is governed by the $d+d'$-dimensional
stochastic differential equation
\begin{equation*}
  {\cal X}^\theta_t={\cal X}^\theta_0+\int_0^t  {\cal A}({\cal X}^\theta_s,\theta) \dd B_s+ \int_0^t {\cal B}({\cal X}^\theta_s) \dd s,
\end{equation*}
where only a discrete sampling of the first $d$ components is observable. This is left to further research.
\subsection{Main results}
Before going into the details of our results, we present a very
simple example which gives some insight
on the type of results that one can
expect.
\begin{example}[Multiplicative Brownian case]\label{ex:multiplicative}
  Assume that the model is \[X^\theta_t=\theta B_t\] (corresponding to
  $b\equiv 0$ and $a(\cdot,\theta)=\theta, \xi_0=0$).
  \begin{enumerate}
  \item Consider a first situation where one observes the diffusion at
    discrete times. Hence, the observation is $(X_{i/n})_{0\leq i\leq
      n}$, or equivalently $(Z_i=\theta (B_{i/n}-B_{{(i-1)/n}})=\theta
    G_i)_{1\leq i\leq n}$, where $G_i$ are independent centered
    Gaussian variables, with a known variance. Thus, the estimation of
    $\theta^2$ is achieved at rate $\sqrt n$, with a minimal variance
    equal to $2 \theta^4$.
  \item Now consider a second situation where one observes only the
    integrated diffusion at discrete times. Hence, the observation is
    $(\bar X_{i}=\theta \int_{0}^{1} B_{\frac {(s+i)} n}
    \mu(ds)=\theta G'_i)_{1\leq i\leq n}$, where $(G'_i)_i$ is a
    centered Gaussian vector, with a known covariance matrix. In
    addition, this matrix is invertible and thus, $\theta^2$ can be
    estimated with the same rate and asymptotic variance as before.
  \end{enumerate}
  This means that observing the process at discrete times or its
  integrated version lead to the same accuracy in the parameter estimation. The
  results of this paper state that this is true, even for the more
  general models \eqref{E:EDS_theta}-\eqref{E:cond_init_xi}, which is
  far from intuitive.
\end{example}
Before stating our main results, we define the working assumptions of this paper. The coefficients
 $a \: : \R \times \Theta \to \R$ and $ b \: : \R \to\R$, are assumed to satisfy the following set of conditions (as usual, derivatives w.r.t. $\theta$ are denoted with a dot: for instance, $\partial_\theta a=\dot a$).\\
{\bf Assumption (R)}
\begin{enumerate}
\item [1)] The function $a \: : \R \times \Theta \to \R$ is $\mathcal{C}^{1+\gamma}$ for some
  $\gamma \in (0,1)$ (it admits a derivative which is $\gamma$-H{\"{o}}lder). The one dimensional functions
   $x \mapsto a(x,\theta)$, $x
  \mapsto \dot a(x,\theta)$, $x \mapsto b(x)$ are assumed
  to be $\mathcal{C}^3(\mathbb{R})$.
\item [2)] The functions $a$, $\dot a$ and $b$ and all
  their derivatives with respect to $x$ are bounded uniformly in
  $\theta$.
\item [3)] We have the non degeneracy condition, for some
  $\underline{a}$: $a(x,\theta)>\underline{a}>0$ for all $x,\theta$.
\end{enumerate}
Actually, the uniform controls in {\bf (R)} can be weakened to local
ones, using extra techniques of space localization (see Lemma 4.1 in
\cite{Gobet1}). We omit further details. An extension of our results
to a multidimensional parameter $\theta$ and to time dependent
coefficients is straightforward, in the same way as it is done in
\cite{Gobet1} and \cite{Gobet2}.

We denote by $\mathbb{P}^\theta$ the law on $\mathcal{C}([0,1])$ of
the process $X^\theta$, and then simply denote $X$ the canonical
process on $\mathcal{C}([0,1])$. We let $\bf{p}^{n,\theta}$ denote the law on $\mathbb{R}^{n}$ of the
observation $\mathcal{O}^n:=(\Xb_j)_{j=0,\dots,n-1}$, when the true
value of the parameter is $\theta$. And for $\theta_0$, $\theta_1$ two
values of the parameter we introduce the likelihood ratio,
\begin{equation}
Z^n_{\theta_0,\theta_1}=\frac{\dd \bf{p}^{n,\theta_1}}{ \dd \bf{p}^{n,\theta_0}}(\mathcal{O}^n).
\end{equation}

The main result is that this statistical model satisfies the so called
LAMN property. For this denote the sequence $u_n:=n^{-1/2}$, and let
$\theta_0 \in \Theta$ and $h \in \mathbb{R}$ such that $\theta_0+u_nh
\in \Theta$, $\forall n$.  Then, by the following theorem, the model
has the LAMN property for the likelihood at point $\theta_0$, with
rate $u_n$ and conditional information:
$$
\mathcal{I}_{\theta_0}=2\int_0^1
\left(\frac{\dot{a}}{a}\right)^2(X_s,\theta_0) \dd s.
$$
\begin{thm}  Assume \ref{regulariteR},
\label{T:LAMN_observation}
then we have the expansion,
$$
\log Z^n_{\theta_0,\theta_0+u_nh}=h N_n-1/2 h^2 I_n+R_n,
$$
where $ I_n \xrightarrow[n \to \infty]{\mathbb{P}^{\theta_0}}
\mathcal{I}_{\theta_0} $, $R_n \xrightarrow[n \to
\infty]{\mathbb{P}^{\theta_0}} 0$ and there exists an extra random
variable $N\sim \mathcal{N}(0,1)$ independent of the process $X$
such that, $N_n$ converges in law under
$\mathbb{P}^{\theta_0}$ to $N \sqrt{\mathcal{I}_{\theta_0}}$.

Moreover this convergence is stable: for any random variable $F$
measurable with respect to the $X$, we have the convergence in law
$(F,N_n) \xrightarrow[n \to \infty]{{law}} (F,N
\sqrt{\mathcal{I}_{\theta_0}})$. In particular it implies the joint
convergence under $\mathbb{P}^{ \theta_0}$:
$$
(I_n,N_n) \xrightarrow[n \to \infty]{{law}}
(\mathcal{I}_{\theta_0},N \sqrt{\mathcal{I}_{\theta_0}}).
$$
\end{thm}
\begin{rem}
  Let us stress that the rate $u_n=n^{-1/2}$ and the information
  $\mathcal{I}_{\theta_0}$ are the same one as for the pointwise
  observation (see Genon--Catalot and Jacod \cite{VGC}). This
  corroborates the intuition from Example \ref{ex:multiplicative}.
\end{rem}

We will not be able to prove directly this result, instead we shall
consider first the easier problem where one can observe additionally
the exact value of the diffusion at some instants.
This
device was proved to be useful in Gloter and Jacod \cite{GJ1} for
the study of a Gaussian diffusion process observed with noise that
leads to non Markovian observations too.

Let $k=k_n$ be an integer in $\{1,\dots,n\}$ and define $L=L_n:=
\left\lfloor n/k \right\rfloor$, then we consider the set of random
variables:
$$
\mathcal{O}^{n,\text{aug}}=\mathcal{O}^{n} \cup \left\{
  X_{\frac{kl}{n}}, l=1,\dots,L \right\} \cup \left\{ X_1 \right\}.
$$
Since this set of variables contains more data than the initial set,
we call it the {\em augmented observation} set. Clearly, we can
split this set into blocks,
$\Block_0,\dots,\Block_L$, where for $l=0,\dots,L-1$
$$\Block_{l}=\left\{ \Xb_{kl},\dots,\Xb_{kl+k-1},X_{k(l+1)/n}
\right\}$$
and $\Block_{L}=\left\{ \Xb_{kL},\dots,\Xb_{n-1},X_{1}
\right\}$.  Note that if $kL=n$ we consider that the last block is
empty, and (immediate) modifications should take care of this in the
sequel, however to have shorter notations we will not explicitly write
these modifications.

The advantage of this set of augmented observation is that using the
Markov property of $X$, the law the block $\Block_l$ conditional to
the previous blocks $(\Block_{l'})_{l' <l}$ only depends on the last
variable, $X_{\frac{kl}{n}}$, of the block $\Block_{l-1}$.

Denote by $\bf{p}^{n,\aug,\theta}$ the law of $\mathcal{O}^{n,\aug}$
on $\R^{n+L+1}$ and introduce the likelihood ratio for the augmented
observation:
\begin{equation}
Z^{n,\aug}_{\theta_0,\theta_1}=\frac{\dd \bf{p}^{n,\aug,\theta_1}}
{ \dd \bf{p}^{n,\aug,\theta_0}}(\mathcal{O}^{n,\aug}).
\end{equation}
\begin{thm}
\label{T:LAMN_sur_observation}
There exists a sequence $k_n \to \infty$, such that the {\em augmented
  model} satisfies the LAMN property:
$$
\log Z^{n,\aug}_{\theta_0,\theta_0+u_nh}=h N^\aug_n-1/2 h^2
I^\aug_n+R^\aug_n,
$$
where $ I^\aug_n \xrightarrow[n \to \infty]{\mathbb{P}^{\theta_0}}
\mathcal{I}_{\theta_0} $, $R^\aug_n \xrightarrow[n \to
\infty]{\mathbb{P}^{\theta_0}} 0$ and there exists an extra random
variable $N\sim \mathcal{N}(0,1)$ independent of the process $X$
such that, $N^\aug_n$ converges in law under $\mathbb{P}^{\theta_0}$
to $N \sqrt{\mathcal{I}_{\theta_0}}$. Moreover this convergence is
stable.
\end{thm}
>From Theorem \ref{T:LAMN_sur_observation} and from the consequences
of the LAMN property, an {\em asymptotically optimal} estimator
$\theta_n$ in the augmented model should satisfy that
$\sqrt{n}(\theta_n-\theta_0)$ is asymptotically distributed under
$\mathbb{P}_{\theta_0}$ as $\frac{1}{\sqrt{\mathcal{I}_{\theta_0}}}
N$.  However any estimator in the initial model of observation
$\mathcal{O}^n$ can be seen as an estimator in the augmented model,
hence the Theorem \ref{T:LAMN_sur_observation} is sufficient by
itself to imply a lower bound for estimation in the initial model.
\begin{rem} The fact that $k_n \to \infty$ means that the data added in the observation
  are sparse compared to the initial data. Actually, the Theorem
  \ref{T:LAMN_sur_observation} holds for any sequence $k_n$ whose
  growth to $\infty$ is slow enough.
\end{rem}
If we assume now that $k_n=k \in \mathbb{N}$ remains fixed as $n \to
\infty$, the number of data $(X_{\frac{kl}{n}})_{l=0,\dots,L}$ added
to the model
is not negligible compared to the number of initial data. Hence the
statistical properties of the augmented model shall depend on $k$
and thus differ from the statistical properties of the initial model
given in Theorem \ref{T:LAMN_observation}. Actually we have the
following LAMN property for the augmented model in that case.
\begin{thm}\label{T:LAMN_k_fixe}
  If a sequence $k_n=k$ is fixed, then the {\em augmented model}
  satisfies the LAMN property with rate $u_n=n^{-1/2}$ and conditional information equal to:
  $$
  \mathcal{I}_{k,\theta_0}=2\left(\frac{k+1}{k}\right) \int_0^1
  \left(\frac{\dot{a}}{a}\right)^2(X_s,\theta_0) \dd s.
  $$
\end{thm}
As expected, the conditional information is greater by a factor
$(k+1)/k$, due to the non negligibility of the added observations.
Actually this factor should be read as $1+\frac 1k$, meaning that an
addition of $\frac 1k \%$ of data increases the information in the
same way. Local means and values at discrete points are not
redundant (as expected from the multiplicative Brownian case, see
Example \ref{ex:multiplicative}) and moreover, they bring an equal
information. Considering $k=1$ is interesting, since we observe then
on each block $[i/n,(i+1)/n]$ both the exact value $X_{\frac{i}{n}}$
and a mean $\Xb_i$. It appears that the asymptotic information is
then twice the information given by the observation of only the
exact values (or only the means).

\subsection{Outline of the paper}
In Section \ref{S:Score} we study the score function given by the
observation of only one block of data ($\Block_0$ for instance). We
first focus on the existence of a density for a block of data; and in
the case of a block of size 2, $(n^{1/2} \int_0^{1} (X_{s/n}-X_0)
  \dd \mu(s), n^{1/2} (X_{\frac{1}{n}}-X_0) )$ we give original
lower and upper bounds of Gaussian type for the density. It is useful
for our proof of the LAMN property, but it is also interesting for
itself.

In Section \ref{Ss:An_exact_expression} we present an exact expression
for the score function of a block of data $\Block_0$ (see Theorem
\ref{T:likehood_block}).  This result is the key point in the proof of
the LAMN property, it extends a former result of Gobet \cite{Gobet1}
\cite{Gobet2} which gave the score function for the observation of
$X^\theta_{\frac{1}{n}}$.  In Section \ref{S:A_Gaussian_approximation}
we study an explicit approximation for the score function when the
sampling interval $1/n$ tends to zero and the length of the block
$k/n$ remains moderate so that one can consider the coefficients of
the diffusion $X$ almost constant on the interval $[0,k/n]$.  The key
point is the Gaussian approximation for the diffusion given in Section
\ref{S:approximation_for_diffusion}.

In Section \ref{S:Asymptotic_study} we deduce from the previous
results a proof of Theorems
\ref{T:LAMN_sur_observation}--\ref{T:LAMN_k_fixe} and Section
\ref{S:LAMN_property} deals on how to deduce Theorem
\ref{T:LAMN_observation} from Theorem \ref{T:LAMN_sur_observation}.

Finally the Appendix contains the proof of some results of Section
\ref{S:The_density} together with some useful lemmas.\\

{\bf Notations.} In our proofs, we will keep the same notation for constants which may
change from one line to another. In particular, the constants $c,
c(k), c(p), c(p,k)$ will stand for all finite, non-negative and
non-decreasing deterministic functions of an index $p$ (arising from
$\mathbf{L}^p$-norm) and of the block size $k$. These constants are
independent of $n, \theta$ and depend on the process $X^\theta$, only
through the bounds on the coefficients $a, b$ and their derivatives.

\section{Score function for a block of data} \label{S:Score}
In this section we shall study the law of the blocks of data $\Block_l$; recalling
the Markov property of the process $X$ it is sufficient to focus on $\Block_0=
\left\{ \Xb_{0},\dots,\Xb_{k-1},X_{k/n}  \right\}$ assuming that
the diffusion $X$ now starts from some value $x_0$.
In this section it is convenient to transform the short time
asymptotic $k/n \to 0$ into an almost stationarity property of the coefficients.
To this end, we introduce the rescaled process $\xr^{n,\theta}_t=n^{\frac{1}{2}}(X^\theta_\frac{t}{n}-x_0 )$ (where $X^\theta$ solves \eqref{E:EDS_theta} with $X_0^\theta=x_0$). It solves the equation
\begin{equation}
\label{E:defXronde}
\dd \xr_t^{n,\theta}=a_n(\xr_t^{n,\theta}, \theta) \dd W_t+b_n(\xr_t^{n,\theta})\dd t,
\quad \xr_0^{n,\theta}=0,
\end{equation}
where $W$ is a standard Brownian motion (arising from the rescaling of $B$),  and
\begin{equation}\label{E:expression_an_bn}
a_n(x,\theta)=a(x_0+n^{-1/2}x,\theta),\quad  b_n(x)=n^{-1/2}b(x_0+n^{-1/2}x).
\end{equation}
Since for the score we are only concerned with the law of $\xr^{n,\theta}$, we can assume that $W$ is independent of the rescaling coefficient $n$.

\subsection{The density of an integrated diffusion} \label{S:The_density}
In this section, we will present preliminary results on the density of the law of the mean of a diffusion process. However the proofs are postponed to Section \ref{S:Proof_of}. To our knowledge, the lower and upper bounds for this density are new results.

\subsubsection{Existence of the density} \label{S:Existence_of}
Our first result actually deals with the two dimensional variable
given by solely one local mean and the exact value:
\begin{align}\label{E:coupleUV}
  (U^{n,\theta},V^{n,\theta}):&= \left( \int_0^1 \xr_s^{n,\theta} \dd
    \mu(s), \xr_1^{n,\theta} \right) \\ \nonumber &\overset{law}{=}
  \left(n^{1/2} \int_0^{1} (X_{\frac{s}{n}}^\theta-x_0) \dd \mu(s),
    n^{1/2} (X^\theta_{\frac{1}{n}}-x_0) \right).
\end{align}
Notice that, by the Markov property, the preliminary study of this
bi--dimensional variable will be a key step to obtain results on the
observation vector $\mathcal{O}^n$.

\begin{thm}\label{T:densiteUV}
  Assume \ref{regulariteRprime}, then the vector
  $(U^{n,\theta},V^{n,\theta})$ admits a density
  $\pUV^n_{x_0}(.,.,\theta)$ on $\R^2$, and there exist two constants
  $c_1>c_2>0$, such that,
\begin{equation}\label{E:maj_min_densite}
c_1^{-1}e^{-c_1(u^2+v^2)} \le \pUV^n_{x_0}(u,v,\theta) \le
c_2^{-1}e^{-c_2(u^2+v^2)}.
\end{equation}
The constants $c_1$ and $c_2$ only depend on the bounds on the
coefficients $a$, $b$ and their derivatives.
\end{thm}
The proof of this theorem is given in Section \ref{S:Proof_of}.  The
existence of the density is obtained by means of the Malliavin
calculus.  On the other hand, the upper and lower bounds rely on the
direct study of $(U^{n,\theta},V^{n,\theta})$ around its skeleton
(see Hirsch and Song \cite{Hirsch1} \cite{Hirsch2} for related
works; and Kohatsu--Higa \cite{Kohatsu} for different methods
involving Malliavin calculus).

The following is a direct corollary of Theorem \ref{T:densiteUV}:
\begin{cor}
  The vector $\Block_0=\left\{ \Xb_{0},\dots,\Xb_{k-1},X_{k/n}
  \right\}$ admits a positive density.
\end{cor}
\begin{proof}
  The bi--dimensional process
  $(\Xb_l,X_\frac{l+1}{n})_{l=0,\dots,k-1}$ is a Markov chain with
  transition density $p_{x_l}(\overline{x}_l,x_{l+1},\theta)=
  n\pUV^n_{x_l}(n^{\frac{1}{2}}
  (\overline{x}_{l+1}-x_l),n^{\frac{1}{2}}(x_{l+1}-x_l),\theta)$. Then
  it is clear that the vector $\Block_0$ admits a positive density.
\end{proof}

\subsubsection{Invertibility of the Malliavin covariance matrix of a block}\label{S:Notations_from}

Actually the existence of a density for the law of the random variable
$\Block_0$ will not be sufficient, and we need a non degeneracy
condition for this variable.

Before this, let us precise briefly a few notations from the
Malliavin calculus, used in the sequel (see Nualart \cite{Nualart1}
\cite{Nualart2} for details).  We let $H$ be the Hilbert space
$\mathbf{L}^2([0,\infty))$ so that the Brownian motion $(W_t)_{t \in
  [0,\infty)}$, appearing in \eqref{E:defXronde}, is canonically associated to
this Hilbert space via the standard $\mathbf{L}^2$ isometry.  In
this setting, for any $p \ge 1$ and natural number $q$, recall that
the set $\D^{q,p}$ denotes the space of real valued Wiener
functionals with $q$ derivatives and whose derivatives belong to
$\mathbf{L}^p(\Omega)$. If we denote by $D$ the derivative operator
then the space $\D^{q,p}$ is endowed with the norm, $ \normeDqp{F}=
\left[ E(\abs{F}^p)+\sum_{j=1}^q E(\norme{D^j
    F}_{\mathbf{L}^2([0,\infty)^j)}^p) \right]^{\frac{1}{p}}.$
The space of variable with $q$ derivatives in any
$\mathbf{L}^p(\Omega)$ is denoted $\D^{q,\infty}=\cap_{p \ge 1}
\D^{q,p}$.  These definitions can be extended to random variables with
values in any Hilbert space $V$ and the corresponding spaces are
denoted $\D^{q,p}(V)$, $\D^{q,\infty}(V)$ (see Section 1.5 in Nualart
\cite{Nualart1}).  In particular the operator $D$ is then well defined
from $\D^{q,\infty}$ to $\D^{q-1,\infty}(H)$.
Finally, the adjoint operator of $D$ is the Skohorod integral $\delta$,
and the Malliavin covariance matrix of an element $F \in
\D^{1,\infty}(\mathbb{R}^d)$ is defined as the matrix
$\gamma_{F_1,\cdots,F_d}=[\psH{D.F_i}{D.F_j}]_{1\le i,j \le d}$.

Now, we consider the variables,
\begin{align} \label{E:defU0}
  U_0^{n,\theta}&:= \int_0^1 \xr_s^{n,\theta} \dd \mu(s)  \\
  U_1^{n,\theta}&:= \int_0^1 (\xr_{s+1}^{n,\theta}-\xr_{s}^{n,\theta})
  \dd \mu(s)\label{E:defU1}\\
& \quad \quad \vdots \nonumber\\
U_{k-1}^{n,\theta}&:=\int_0^1
(\xr_{s+k-1}^{n,\theta}-\xr_{s+k-2}^{n,\theta}) \dd \mu(s)\label{E:defUkmoins}\\
\label{E:defUk} U_k^{n,\theta}&:=\int_0^1
(\xr_{k}^{n,\theta}-\xr_{s+k-1}^{n,\theta}) \dd \mu(s).
\end{align}
Note that the joint law of these $k+1$ variables is, by rescaling,
the same as the law of the vector composed with variables of the first
block $\Block_0$: $n^{\frac{1}{2}}( \Xb_0^\theta-x_0, \Xb_1^\theta-\Xb_0^\theta ,
\dots, \Xb_{k-1}^\theta-\Xb_{k-2}^\theta,
X_{\frac{k}{n}}^\theta-\Xb_{k-1}^\theta )$.
These variables satisfy the following non degeneracy property whose
proof is postponed to Section \ref{S:Proof_of_Proposition}.
\begin{prop}\label{P:inversibleMalliaB}
  Under \ref{regulariteRprime},
  $(U_0^{n,\theta},\dots,U_{k}^{n,\theta})\in
  \D^{3,\infty}$. Denote by $K(\theta)$ the
  Malliavin covariance matrix of $(U_0^{n,\theta},\dots,U_{k}^{n,\theta})$. It is $a.s.$ an invertible matrix and for all $p\geq 1$, we have
  $$
  E\left( \abs{\det(K(\theta))}^{-p} \right) \le c(p,k).
  $$
\end{prop}

\subsection{An exact expression using Malliavin calculus}
\label{Ss:An_exact_expression}
In this Section we intend to give an exact expression for the score
function of the observation of $\Block_0$ or equivalently for the
vector $(U^{n,\theta}_0,\dots,U^{n,\theta}_k)$ given by
\eqref{E:defU0}--\eqref{E:defUk}.

Under the condition \ref{regulariteRprime}, we know that there exists
a version of the solution of \eqref{E:defXronde} such that $P-$almost
surely the function $\theta \to \xr^\theta_t$ is continuously
differentiable for all $t$ and $\tau^{n,\theta}_t:= \frac{\partial
  \xr^{n,\theta}_t}{\partial \theta}$ is a solution of the stochastic
equation (see Kunita \cite{Kunita}):
\begin{align}
\label{E:def_tau}
\dd \tau^{n,\theta}_t&= \frac{\partial a_n}{\partial
  x}(\xr^{n,\theta}_t,\theta)\tau^{n,\theta}_t \dd W_t +
\frac{\partial a_n}{\partial \theta}(\xr^{n,\theta}_t,\theta)\dd W_t +
\frac{\partial b_n}{\partial x}(\xr^{n,\theta}_t)\tau^{n,\theta}_t \dd
t
\\
\nonumber \quad \tau^{n,\theta}_0&=0.
\end{align}
The main result of this section is an explicit representation for the
derivative of the log-likelihood of one block.  This extends a former
result given by Gobet (see \cite{Gobet1} \cite{Gobet2}).
\begin{thm}\label{T:likehood_block} The random vector
  $(U_0^{n,\theta},\dots,U_k^{n,\theta})$ admits a positive density on
  $\mathbb{R}^{k+1}$, denoted by $p_{x_0}(u_0,\dots,u_k,\theta)$. For
  a.e. $(u_0,\dots,u_k)$, this density is an absolutely continuous
  function with respect to the parameter $\theta$ and we have the
  formula:
\begin{equation*}
\frac{\dot{p}_{x_0}}{p_{x_0}}(u_0,\dots,u_k,\theta)
=E\left[ \delta \left(
\sum_{0\le j,j' \le k}
\frac{\partial U^{n,\theta}_j}{\partial \theta} K(\theta)^{-1}_{j,j'}
D. U^{n,\theta}_{j'}
\right) \mid (U^{n,\theta}_j=u_j)_{j=0,\dots,k}
\right],
\end{equation*}
where $K(\theta)^{-1}$ is the inverse of the Malliavin covariance
matrix of $(U_0^{n,\theta},\dots,U_k^{n,\theta})$.
\end{thm}
\begin{proof}
  Denote $U^{n,\theta}$ the Wiener functional,
  $U^{n,\theta}=(U_0^{n,\theta}, \dots,U_k^{n,\theta})$ and let
  $f:\mathbb{R}^{k+1} \to \mathbb{R}$ be a smooth function with
  compact support. Then the function $\theta \mapsto
  E\left[f(U^{n,\theta}) \right]$ can be differentiated pointwise and:
\begin{equation*}\label{E:derive_perm}
\frac{\partial}{\partial \theta} E\left[f(U^{n,\theta})\right]
=E\left[\sum_{j=0}^k \frac{\partial f}{\partial u_j}(U^{n,\theta})
\frac{\partial U_j^{n,\theta}}{\partial \theta}
\right].
\end{equation*}
By Proposition \ref{P:inversibleMalliaB} the Malliavin covariance
matrix of $U^{n,\theta}$ is invertible and a standard computation on
Wiener functionals (see formula (2.4) p.81 in Nualart \cite{Nualart1})
shows that: $ \frac{\partial f}{\partial u_j}(U^{n,\theta})=
\sum_{j'=0}^k \psH{D(f(U^{n,\theta}))}{DU_{j'}^{n,\theta}}
K(\theta)^{-1}_{j,j'}.  $ It follows that $\frac{\partial}{\partial \theta}  E\left[f(U^{n,\theta})\right]$ is equal to
\begin{align}
  E\left[\sum_{j=0}^k \sum_{j'=0}^k
    \psH{D(f(U^{n,\theta}))}{DU^{n,\theta}_{j'}} K(\theta)^{-1}_{j,j'}
    \frac{\partial U_j^{n,\theta}}{\partial \theta} \right]
   \nonumber &=E\left[\psH{D(f(U^{n,\theta}))}{L^\theta} \right]
\end{align}
where $L^\theta$ is the $H$-valued random variable:
\begin{equation}\label{E:defL}
L^\theta:=\sum_{j=0}^k
\sum_{j'=0}^k
\frac{\partial U_j^{n,\theta}}{\partial \theta}
K(\theta)^{-1}_{j,j'}
DU^{n,\theta}_{j'}.
\end{equation}
Introducing $\delta$ the adjoint operator of $D$, we get
\begin{align}
\label{E:derivEfU}
\frac{\partial}{\partial \theta} E\left[f(U^{n,\theta})\right]
&=E\left[ f(U^{n,\theta}) \delta( L^\theta) \right].
\end{align}
Let $g$ be any smooth function with compact support on $\mathbb{R}$.
Using the integration by part formula and the equation
\eqref{E:derivEfU} we have:
\begin{align*}
  \int \dd \theta \dot{g}(\theta) E(f(U^{n,\theta}))&=-\int \dd \theta
  g(\theta) \frac{\partial}{\partial \theta}
  E\left[f(U^{n,\theta})\right] =-\int \dd \theta g(\theta) E\left[
    f(U^{n,\theta}) \delta( L^\theta) \right]
  \\
  &=-\int \dd \theta g(\theta) E\left[ f(U^{n,\theta}) E[ \delta(
    L^\theta) \mid (U^{n,\theta}_0, \dots,U^{n,\theta}_k)] \right].
\end{align*}
Introducing the density of the random vector $U^{n,\theta}$ the
equation above writes,
\begin{multline*}
  \int \dot{g}(\theta) \dd \theta\int f(u_0,\dots,u_k)
  p_{x_0}(u_0,\dots,u_k,\theta) \dd u_0\dots\dd u_k
  \\
  =- \int g(\theta)\dd \theta \int f(u_0,\dots,u_k) E[ \delta(
  L^\theta) \mid (U^{n,\theta}_l=u_l)_l] p_{x_0}(u_0,\dots,u_k,\theta)
  \dd u_0\dots\dd u_k.
\end{multline*}
Now using Fubini's theorem it can be seen that $\dd u_0\dots\dd
u_k$-almost\break everywhere the function $\theta \to
p_{x_0}(u_0,\dots,u_k,\theta)$ is absolutely continuous with
$$\dot{p}_{x_0}(u_0,\dots,u_k,\theta)= E[ \delta( L^\theta) \mid
(U^{n,\theta}_l=u_l)_l] p_{x_0}(u_0,\dots,u_k,\theta).
$$
Hence the theorem is proved.
\end{proof}
Remark that the proof of Theorem \ref{T:likehood_block} does not
rely on the specific expression \eqref{E:defU0}--\eqref{E:defUk} and
thus an analogous representation for the score function seems
achievable in many situations.
%

\subsection{A Gaussian approximation for the log-likelihood}
\label{S:A_Gaussian_approximation}
In this section we  intend to give a tractable approximation
for the score function of $(U^{n,\theta}_0,\dots,U^{n,\theta}_k)$.
\subsubsection{Approximation for the diffusion}
\label{S:approximation_for_diffusion} We introduce
$\tilde{\xr}^{\theta}_t=a(x_0,\theta) W_t$ and
$\tilde{\tau}^\theta_t=\dot{a}(x_0,\theta) W_t$ which stand -by
\eqref{E:defXronde} and \eqref{E:def_tau}- for the first order
approximations of $\xr^{n,\theta}_t$ and $\tau^{n,\theta}_t=
\frac{\partial \xr^{n,\theta}_t}{\partial \theta} $. Then, we
consider the quantities obtained by replacing in
\eqref{E:defU0}--\eqref{E:defUk} the process $\xr$ by this Gaussian
approximation
\begin{align}
\label{E:defUtilde0}
\widetilde{U}^\theta_0&:=a(x_0,\theta) \int_0^1 W_s \dd \mu(s) =a(x_0,\theta) \int_0^{1} \mu([s,1]) \dd W_s,\\
\label{E:defUtildej}
\widetilde{U}^\theta_j&:=a(x_0,\theta) \int_{0}^1 (W_{j+s}-W_{j-1+s})
\dd \mu( s), \quad \text{ for $j=1,\dots,k-1$} \\ \nonumber
&=a(x_0,\theta) \int_{j-1}^{j} \mu([0,s-(j-1)]) \dd W_s +a(x_0,\theta)
\int_{j}^{j+1} \mu([s-j,1]) \dd W_s,\\
\label{E:defUtildek}
\widetilde{U}^\theta_k&:=a(x_0,\theta) \int_{0}^1 (W_{k}-W_{k-1+s})
\dd \mu( s), \\ \nonumber &= a(x_0,\theta) \int_{k-1}^{k}
\mu([0,s-(k-1)]) \dd W_s,
\end{align}
where we have repeatedly used the Fubini theorem for stochastic integrals (see \cite{Revuz} p.176). In the next lemma we control the difference between the
$U_j^{n,\theta}$ and their approximation in terms of Sobolev norm.
\begin{lem}
\label{L:compUUtilde}
For all $k,p>1$, there exist constants $c(k,p)$, $c(p)$ such that for
all $j\in\{0,\dots,k\}$:
\begin{align} \label{E:diffUUtilde1}
  \norme{U^{n,\theta}_j-\widetilde{U}^\theta_j }_{2,p} &\le c(k,p)
  n^{-1/2}, \quad & \norme{\widetilde{U}^\theta_j}_{3,p} \le c(p), \\
  \label{E:diffUUtilde2} \norme{\frac{\partial
      U^{n,\theta}_j}{\partial \theta} -\frac{\partial
      \widetilde{U}^{n,\theta}_j }{\partial \theta} }_{2,p} &\le
  c(k,p) n^{-1/2}, \quad & \norme{\frac{\partial
      \widetilde{U}^\theta_j }{\partial \theta}}_{3,p} \le c(p),
\end{align}
\begin{equation}
\label{E:diffUUtilde4}
\forall \: {0 \le j,j' \le k, \quad }
\abs{ E\left( U^{n,\theta}_j U^{n,\theta}_{j'} -
\widetilde{U}^{\theta}_j \widetilde{U}^{\theta}_{j'} \right)  }
\le c(k)n^{-1}.
\end{equation}
\end{lem}
\begin{proof}
  The inequalities on the right hand side of
  \eqref{E:diffUUtilde1}--\eqref{E:diffUUtilde2} are immediate by the
  definition of $\widetilde{U}^\theta_j$.

  Comparing expressions of \eqref{E:defU0}--\eqref{E:defUk} with
  \eqref{E:defUtilde0}--\eqref{E:defUtildek}, the two remaining bounds
  in \eqref{E:diffUUtilde1}--\eqref{E:diffUUtilde2} will be a
  consequence of the Minkowski inequality - for the Sobolev norm - and
  of the control on the diffusions:
\begin{align*}
  \sup_{t\leq k}\norme{\xr^{n,\theta}_t-\tilde{\xr}^{\theta}_t}_{2,p}
  + \sup_{t\leq
    k}\norme{\tau^{n,\theta}_t-\tilde{\tau}^{\theta}_t}_{2,p} &\le
  n^{-1/2} c(k,p).
\end{align*}
We only prove the control on $\xr^{n,\theta}$ since the proof for
$\tau^{n,\theta}$ is analogous. Recalling
\eqref{E:defXronde}--\eqref{E:expression_an_bn}, we can write
\begin{align}\nonumber
  & \xr^{n,\theta}_t-\tilde{\xr}^{\theta}_t = \int_0^t
  [a_n(\xr^{n,\theta}_s,\theta)-a(x_0,\theta)] \dd W_s+ \int_0^t
  b_n(\xr^{n,\theta}_s) \dd s
  \\
  &\hspace{-5mm}=\frac{1}{\sqrt{n}} \int_0^t \int_0^1 a'_x(x_0 +
  \frac{u\xr^{n,\theta}_s}{\sqrt{n}},\theta) \xr^{n,\theta}_s \dd u
  \dd W_s+ \frac{1}{\sqrt{n}} \int_0^t b(x_0+
  \frac{\xr^{n,\theta}_s}{\sqrt{n}}) \dd s. \label{eq:manu1}
\end{align}
But we know \cite{Nualart1} that under \ref{regulariteRprime} the
variables $\xr^{n,\theta}$ belong to $\D^{3,\infty}$ with a control
(independent of $\theta,n$): $\sup_{u_1,u_2 \le s\le k}
E(|D^2_{u_1,u_2} \xr^{n,\theta}_s |^p ) \le c(p,k)$. This is
sufficient to deduce
$\norme{\xr^{n,\theta}_t-\tilde{\xr}^{\theta}_t}_{2,p} \le
n^{-1/2}c(p,k)$ after a few computations.

To obtain \eqref{E:diffUUtilde4} note that by \eqref{E:diffUUtilde1}
it is sufficient to show $E\left(
  (U^{n,\theta}_j-\widetilde{U}^\theta_j) \widetilde{U}^\theta_{j'}
\right) \le c(k)n^{-1}$. This property will follow again from an
analogous relation on the diffusion,
$$
\sup_{t,t'\le k}\abs{ E\left( ( \xr^{n,
      \theta}_t-\tilde{\xr}^\theta_t ) \tilde{\xr}^\theta_{t'}
  \right)} \le c(k) n^{-1}.
$$
Indeed, from \eqref{eq:manu1}, the above expectation is equal to
\begin{multline*}
  n^{-1/2}\int_0^{t \wedge t'} \int_0^1 E \left[ a'_x(x_0 +n^{-1/2}
    u\xr^{n,\theta}_s,\theta) \xr^{n,\theta}_s\right]
  \dd u a(x_0,\theta_0) \dd s +
  \\
  n^{-1/2} \int_0^{t} E \left[ b( x_0 +
    n^{-1/2}\xr^{n,\theta}_s) W_{t'} \right]a(x_0,\theta)\dd
  s.
\end{multline*}
Using $\abs{ E[a'_x(x_0,\theta) \xr^{n,\theta}_s]} =\abs{ \int_0^s
  a'_x(x_0,\theta) E[ b_n(\xr^{n,\theta}_u) ] \dd u }\le c n^{-1/2}$,
$E \left[ b( x_0) W_{t'} \right]=0$ and the boundedness of $a''_{xx}$
and $b'$, we get the required estimate.
\end{proof}

\subsubsection{Approximation for the log-likelihood}
Let us denote the deterministic tridiagonal matrix $\widetilde{K}$ of
size $(k+1) \times (k+1)$,
\begin{equation*}
\widetilde{K}=
\begin{bmatrix}
  v_1 & c & 0   &0  & 0   \\
  c   & v_1+v_2 & \ddots &0 &0 \\
  0   & \ddots  & \ddots  & \ddots &0 \\
  0 & 0 &\ddots &v_1+v_2 &c \\
  0 &0 & 0 & c & v_2 \\
\end{bmatrix},
\end{equation*}
where the entries of the matrix are:
\begin{eqnarray*}
v_1=\int_0^1 \mu([s,1])^2 \dd s , \quad v_2=\int_0^1 \mu([0,s])^2 \dd s, \quad
c=\int_0^1 \mu([0,s])
\mu([s,1]) \dd s.
\end{eqnarray*}
It can be easily checked that $a^2(x_0,\theta)\widetilde{K}$ is the
covariance matrix of the Gaussian vector
$(\widetilde{U}^\theta_0,\dots,\widetilde{U}^\theta_k)$ and that it
is invertible using \eqref{E:hypomu}.
Now the idea is to introduce the score function that would be
produced from the observation of this Gaussian vector. Hence we let:
\begin{equation}
\label{E:def_Lronde}
\hspace{-4mm}\Lr_{x_0}(u_0,\dots,u_k,\theta)=
\frac{\dot{a}}{a}(x_0,\theta) \left\{
 a(x_0,\theta)^{-2}
\sum_{0\le j ,j'  \le k} u_j \widetilde{K}^{-1}_{j,j'} u_{j'} - (k+1)
\right\}.
\end{equation}
In this section, we will show that this quantity is an approximation
for the true score function $\frac{\dot{p}}{p}$:
\begin{thm}
\label{T:approx_score_gauss}
Let us consider the difference,
\begin{equation}
\label{eq:manu2}
\frac{\dot{p}_{x_0}}{p_{x_0}}(u_0,\dots,u_k,\theta)-
\Lr_{x_0}(u_0,\dots,u_k,\theta):= r_{x_0}(u_0,\dots,u_k,\theta).
\end{equation}
Then we have the following bounds:
\begin{align}\label{E:cond_r_centrage}
  \abs{E\left[r_{x_0}(U^{n,\theta}_0,\dots,U^{n,\theta}_k, \theta)
    \right]} &\le c(k) n^{-1}, \\\label{E:cond_r_moment} \forall p\ge
  1, \quad
  E\left[\abs{r_{x_0}(U^{n,\theta}_0,\dots,U^{n,\theta}_k,\theta)}^{p}\right]^{\frac{1}{p}}
  &\le c(k,p) n^{-1/2}.
\end{align}
\end{thm}

\begin{proof}
  Keeping in mind the definition of $L^\theta$ (see \eqref{E:defL}),
  we introduce its approximation based on the Gaussian quantities
  defined above:
\begin{equation*}
\widetilde{L}^\theta:=\sum_{j=0}^k
\sum_{j'=0}^k
\frac{\partial \widetilde{U}_j^\theta}{\partial \theta}
a(x_0,\theta)^{-2}
\widetilde{K}^{-1}_{j,j'}
D\widetilde{U}^\theta_{j'}.
\end{equation*}
{\bf The first step} is to obtain the following control on the
difference $r_1:=L^\theta-\widetilde{L}^\theta$:
\begin{equation}\label{E:LLtilde}
\forall  p>1,\quad \norme{r_1}_{\mathbb{D}^{1,p}(H)}
\le c(k,p) n^{-1/2}.
\end{equation}
Actually, it is a easy consequence of Lemma \ref{L:compUUtilde},
Proposition \ref{P:inversibleMalliaB} and the invertibility of
$\widetilde K$, noting that the Malliavin covariance matrix of
$\widetilde U^\theta$ coincides with the covariance matrix
$a^2(x_0,\theta) \widetilde K$ of the Gaussian vector $\widetilde
U^\theta$. We omit further details.

\noindent {\bf The second step} is to obtain a simple expression for
$\delta(\widetilde{L}^\theta)$. To see this, we first use the relation
for $F \in \D^{1,\infty}, u\in \D^{1,\infty}(H)$,
$\delta(Fu)=F\delta(u)-\psH{D.F}{u}$ (see \cite{Nualart1}):
\begin{multline*}
  \delta(\widetilde{L}^\theta)=\sum_{j=0}^k \sum_{j'=0}^k
  \frac{\partial \widetilde{U}_j^\theta}{\partial \theta}
  a(x_0,\theta)^{-2}\widetilde{K}^{-1}_{j,j'}
  \delta (D (\widetilde{U}^\theta_{j'})) \\
  - \sum_{j=0}^k \sum_{j'=0}^k a(x_0,\theta)^{-2}
  \widetilde{K}^{-1}_{j,j'} \psH{D.\frac{\partial
      \widetilde{U}_j^\theta}{\partial \theta}} {D.
    \widetilde{U}^\theta_{j'}}.
\end{multline*}
On the one hand, $\delta (D
(\widetilde{U}^\theta_{j'}))=\widetilde{U}^\theta_{j'}$ ($\delta \circ
D$ is the identity operator on the first chaos space). On the other
hand, one has $\frac{\partial \widetilde{U}_j^\theta}{\partial
  \theta}=
\frac{\dot{a}(x_0,\theta)}{a(x_0,\theta)}\widetilde{U}_j^\theta$ by
\eqref{E:defUtilde0}--\eqref{E:defUtildek}. We deduce
\begin{align}\nonumber
  \delta(\widetilde{L}^\theta)&=
  \frac{\dot{a}(x_0,\theta)}{a(x_0,\theta)} \sum_{j=0}^k \sum_{j'=0}^k
  \widetilde{U}_j^\theta a(x_0,\theta)^{-2}\widetilde{K}^{-1}_{j,j'}
  \widetilde{U}^\theta_{j'} \\ & \hspace{3cm}\nonumber
  -\frac{\dot{a}(x_0,\theta)}{a(x_0,\theta)} \sum_{j=0}^k
  \sum_{j'=0}^k a(x_0,\theta)^{-2} \widetilde{K}^{-1}_{j,j'}
  \psH{D.\widetilde{U}_j^\theta} {D. \widetilde{U}^\theta_{j'}} \\
  \nonumber
  &= \frac{\dot{a}(x_0,\theta)}{a(x_0,\theta)} \sum_{j=0}^k
  \sum_{j'=0}^k \widetilde{U}_j^\theta a(x_0,\theta)^{-2}
  \widetilde{K}^{-1}_{j,j'} \widetilde{U}^\theta_{j'}-
  \frac{\dot{a}(x_0,\theta)}{a(x_0,\theta)} (k+1).
\end{align}
Now set
$$
r_2= \frac{\dot{a}(x_0,\theta)}{a^3(x_0,\theta)} \sum_{0\le j,j'
  \le k} \widetilde{U}_j^\theta \widetilde{K}^{-1}_{j,j'}
\widetilde{U}^\theta_{j'} -\frac{\dot{a}(x_0,\theta)}{a^3(x_0,\theta)}
\sum_{0\le j,j' \le k} U_j^{n,\theta} \widetilde{K}^{-1}_{j,j'}
U^{n,\theta}_{j'},
$$
and take the conditional expectation in the relation
$\delta(L^\theta)= \delta(\widetilde{L}^\theta)+\delta(r_1)$: by
Theorem \ref{T:likehood_block}, we get \eqref{eq:manu2} with
$r_{x_0}(u_0,\dots,u_k,\theta)= E\left( \delta(r_1) \mid
  (U^{n,\theta}_j)_j=(u_j)_j \right) + E\left( r_2 \mid
  (U^{n,\theta}_j)_j=(u_j)_j \right)$.

\noindent {\bf The final step} in the proof is to show that
$r_{x_0}$ satisfies conditions
\eqref{E:cond_r_centrage}--\eqref{E:cond_r_moment}.  For the first
condition, since the Skorohod integral has zero mean, we have
$E[r_{x_0}(U^{n,\theta}_0,\dots ,U^{n,\theta}_k,\theta)]=E(r_2)$ and
we conclude using \eqref{E:diffUUtilde4}.

We now prove \eqref{E:cond_r_moment}. The conditional expectation
being a contraction on $\mathbf{L}^p$ it is sufficient to prove
$$
E(\abs{\delta(r_1)}^p)^{\frac{1}{p}} \le c(p,k) n^{-1/2}, \quad
E(\abs{r_2}^p)^{\frac{1}{p}} \le c(p,k) n^{-1/2}.
$$
The first estimate follows from \eqref{E:LLtilde} and the
continuity of the operator $\delta$ from $\D^{1,p}(H)$ to
$\mathbf{L}^p$.  The second one is an immediate consequence of Lemma
\ref{L:compUUtilde}.
\end{proof}
\begin{rem}
  Let us note that the constants $c(k), c(k,p)$ in Theorem
  \ref{T:approx_score_gauss} should increase as the block length $k$
  goes to infinity since the Gaussian approximation ceases to be valid
  in that case.  However in the sequel we shall not need a precise
  evaluation of this dependence on $k$ since we will have the
  possibility to conveniently choose the growth rate of $k=k_n$.
\end{rem}
In the following sections we will need this corollary of Theorem
\ref{T:approx_score_gauss}.
\begin{cor}
\label{C:majdotpp}
We have for all $p>1$,
$$
E\left[ \abs{
    \frac{\dot{p}_{x_0}}{p_{x_0}}(U^{n,\theta}_0,\dots,U^{n,\theta}_k,\theta)}^p\right]
\le c(k,p).
$$
\end{cor}
\begin{proof}
  By Theorem \ref{T:approx_score_gauss} it is sufficient to show that
  $ E\left[ \abs{
      \Lr_{x_0}(U^{n,\theta}_0,\dots,U^{n,\theta}_k,\theta)}^p\right]
  \le c(k,p).$ But from the expression of $\Lr_{x_0}$, this estimate
  is clear.
\end{proof}

\section{Asymptotic study for the augmented model} \label{S:Asymptotic_study}

In this paragraph we establish Theorem \ref{T:LAMN_sur_observation}.
Let us recall some notations: we now deal with the diffusion given by
\eqref{E:EDS_theta}--\eqref{E:cond_init_xi}; $k_n$ is some integer in
$\{1,\dots,n\}$, $L_n=\left\lfloor n/k_n \right\rfloor$ and our
observation consists of the $L_n+1$ blocks $\Block_0,\dots,
\Block_{L_n}$ described in Section \ref{S:Main}. The length of the
block $\Block_l$ is $k_{n,l}+1$, where $k_{n,l}=k_n$ if $l\le L_n-1$
and $k_{n,L_n}=n-L_n k_n$. For sake of simplicity in the sequel we
sometimes omit the dependence with respect to $n$ and
$l$ of the block size, and let $k_{n,l}=k$ with a slight abuse of notation in particular
for the last block of data.

To be able to use the results of the Section \ref{S:Score}, we
introduce on each block the random variables corresponding to the
definitions \eqref{E:defU0}--\eqref{E:defUk} for the first block.
Hence for $l \in \{0,\dots,L_n\}$, we define the $k_{n,l}+1$
following variables:
\begin{align*}
  U_{0,l}&=n^{\frac{1}{2}}( \Xb_{k l}-X_{\frac{kl}{n}} ),
  \\
  U_{1,l}&=n^{\frac{1}{2}}( \Xb_{k l+1}-\Xb_{k l} ),
  \\
  & \: \: \: \vdots \nonumber
  \\
  U_{k-1,l}&=n^{\frac{1}{2}}( \Xb_{kl+k-1}-\Xb_{kl+k-2} ),
  \\
  U_{k,l}&=n^{\frac{1}{2}}( X_{\frac{k(l+1)}{n}}-\Xb_{kl+k-1}).
\end{align*}
Clearly the observation of the $(U_{j,l})$ for $l\in \{0,\dots,L_n\},
j \in \{0,\dots,k_{n,l}\}$ is equivalent to the observation of the
$L_n+1$ blocks.  Using the Markov property for the process $X$ it
appears that the law of the vector $(U_{j,l})_{j=0,\dots,k_{n,l}}$
conditionally to all the variables $U_{j,l'}$ with $l'<l, j \in
\{0,\dots,k_{n,l'}\}$ is the same as conditionally to $X_{kl/n}$ only;
moreover this law - conditionally to $X_{kl/n}=x_0$ - coincides with
that of the vector $(U^{n,\theta}_0,\dots, U^{n,\theta}_k)$ studied in
Section \ref{S:Score}.  Thus it admits the density
$p_{X_{\frac{kl}{n}}}(u_0,\dots,u_k,\theta)$ studied in Sections
\ref{Ss:An_exact_expression}--\ref{S:A_Gaussian_approximation}.  Hence
the log-likelihood of the augmented model admits the additive
structure:
\begin{align*}
  \ln (Z^{n,\aug}_{\theta_0,\theta_0+u_nh}(\mathcal{O}^{n,\aug}))
  &=\sum_{l=0}^{L_n} \ln
  \frac{p_{X_{\frac{kl}{n}}}(U_{0,l},\dots,U_{k,l},\theta_0+u_n
    h)}{p_{X_{\frac{kl}{n}}}(U_{0,l},\dots,U_{k,l},\theta_0)}
  \\
  &=\sum_{l=0}^{L_n} \int_{\theta_0}^{\theta_0+u_n h}
  \frac{\dot{p}_{X_{\frac{kl}{n}}}(U_{0,l},
    \dots,U_{k,l},s)}{p_{X_{\frac{kl}{n}}}(U_{0,l},\dots,U_{k,l},s)}
  \dd s.
\end{align*}
Owing to Theorem \ref{T:approx_score_gauss}, we deduce the
decomposition
\begin{multline*}
  \ln (Z^{n,\aug}_{\theta_0,\theta_0+u_nh}(\mathcal{O}^{n,\aug}))
  =\sum_{l=0}^{L_n} \int_{\theta_0}^{\theta_0+u_n h} \Lr_{
    X_{\frac{kl}{n}} }
  (U_{0,l},\dots,U_{k,l},s) \dd s \\
  + \sum_{l=0}^{L_n} \int_{\theta_0}^{\theta_0+u_n h} r_{
    X_{\frac{kl}{n}} } (U_{0,l},\dots,U_{k,l},s) \dd s.
\end{multline*}
In the above decomposition, we will show in Sections
\ref{S:The_explicit}--\ref{S:The_negligible} that the explicit term
involving $\Lr_{x_0}$ governs the asymptotic behavior of the
log-likelihood ratio; the other term does not contribute in the limit.
\subsection{Proof of Theorem \ref{T:LAMN_sur_observation}: the explicit term}\label{S:The_explicit}
Let us introduce a slight modification of $\Lr_{X_{\frac{kl}{n}}}$,
which has the advantage of being a smoother function w.r.t.
$\theta$:
\begin{equation}\label{E:def_xi}
\xi_{l,n}(\theta)= \frac{\dot{a}}{a}(X_{\frac{kl}{n}},\theta_0)
\left\{
 a(X_{\frac{kl}{n}},\theta)^{-2}
\sum_{0\le j ,j'  \le k} U_{j,l} \widetilde{K}^{-1}_{j,j'} U_{j',l}
- (k+1) \right\},
\end{equation}
and we set $N_n^{\aug}=u_n\sum_{l=0}^{L_n} \xi_{l,n}(\theta_0)$ and
$I_n^{\aug}=-u^2_n\sum_{l=0}^{L_n} \frac{\partial
\xi_{l,n}}{\partial
  \theta}(\theta_0)$.
\begin{prop}\label{P:termPal}
  If $k_n \to \infty$ slowly enough,
\begin{equation}
\label{E:term_LAMN} \sum_{l=0}^{L_n} \int_{\theta_0}^{\theta_0+u_n
h} \Lr_{ X_{\frac{kl}{n}} } (U_{0,l},\dots,U_{k,l},s) \dd s = h
N_n^\aug - \frac{h^2}{2} I_n^\aug+ R_n,
\end{equation}
where $ I^\aug_n \xrightarrow[n \to \infty]{\mathbb{P}^{\theta_0}}
\mathcal{I}_{\theta_0} $, $R_n \xrightarrow[n \to
\infty]{\mathbb{P}^{\theta_0}} 0$ and there exists an extra random
variable $N\sim \mathcal{N}(0,1)$ independent of the process $X$
such that, $N^\aug_n$ converges stably in law under
$\mathbb{P}^{\theta_0}$ to $N \sqrt{\mathcal{I}_{\theta_0}}$.
\end{prop}
\begin{proof}
  Comparing \eqref{E:def_Lronde} with the definition of
  $\xi_{l,n}(\theta)$ above and using a Taylor expansion for
  $\xi_{l,n}(\theta)$ around $\theta_0$, we get the equation
  \eqref{E:term_LAMN} with a remainder term $R_n=R_n^{(1)}+R_n^{(2)}$
  satisfying:
\begin{align}\label{E:resteRn1}
  R_n^{(1)} &= \sum_{l=0}^{L_n}\int_{\theta_0}^{\theta_0+u_n h}
  [\frac{\dot{a}}{a}(X_\frac{kl}{n},s)-\frac{\dot{a}}{a}(X_\frac{kl}{n},\theta_0)]
  \left\{ \sum_{0\le j ,j' \le k} \frac{U_{j,l}
      \widetilde{K}^{-1}_{j,j'} U_{j',l}} { a(X_\frac{kl}{n},s)^2} -
    (k+1) \right\} \dd s, \\ \label{E:resteRn2} \abs{R_n^{(2)}} &\le c
  \sum_{l=0}^{L_n} u_n^{2+\gamma} \left\{ \sum_{0\le j ,j' \le k}
    \abs{U_{j,l} \widetilde{K}^{-1}_{j,j'} U_{j',l}} \right\}
\end{align}
(for $R_n^{(2)}$ we have used that $\theta \mapsto \dot{a}(x,\theta)$
is $\gamma$-H\"{o}lder continuous). To complete the proof, we
repeatedly use the following classical convergence result about
triangular arrays of random variables.
\begin{lem}[Genon-Catalot and Jacod \cite{VGC}, Lemma 9]\label{lemma:9} Let $(\chi^n_l)_{0\leq l\leq L_n}$, $U$ be random variables, with
  $\chi^n_l$ being ${\cal F}^n_{l+1}$-measurable. The two following
  conditions imply
  $\sum_{l=0}^{L_n}\chi^n_l\stackrel{\PP}{\rightarrow} U$:
\begin{eqnarray*}
\sum_{l=0}^{L_n} E\left[\chi^n_l|{\cal F}^n_{l}\right]
\stackrel{\PP}{\rightarrow}U \qquad\mbox{and}\qquad
\sum_{l=0}^{L_n} E\left[(\chi^n_l)^2|{\cal F}^n_{l}\right]
\stackrel{\PP}{\rightarrow}0.
\end{eqnarray*}
\end{lem}

$\bullet$ \underline{We first focus on $N_n^\aug$.} Let us introduce
the sigma field $\F^n_l= \sigma(X_0;B_s, s \le \frac{kl}{n})$ for
$l=0,\dots,L_n$ and $\F^n_{L_n+1}= \sigma(X_0;B_s, s \le 1)$. Then the
variable $\xi_{l,n}(\theta_0)$ is $\F^n_{l+1}$--measurable and the
asymptotic behavior of $N_n^\aug$ will follow from Lemma
\ref{lemma:9}.  To make clearer this point we introduce the following
approximation based on conditionally Gaussian variables:
\begin{equation}\label{E:defxitilde}
\widetilde{\xi}_{l,n}(\theta)=
\frac{\dot{a}}{a}(X_\frac{kl}{n},\theta_0) \left\{
a(X_\frac{kl}{n},\theta)^{-2} \sum_{0\le j ,j'  \le k} \widetilde{U}_{j,l}
\widetilde{K}^{-1}_{j,j'}  \widetilde{U}_{j',l} - (k+1)
\right\}.
\end{equation}
Here, $\widetilde{U}_{j,l}$ is the Gaussian approximation under
$\PP^\theta$ of $U_{j,l}$ corresponding on the block $\Block_l$ to the
variables \eqref{E:defUtilde0}--\eqref{E:defUtildek} on the block
$\Block_0$:
\begin{align*}
  \widetilde{U}_{0,l}&:=a(X_\frac{kl}{n},\theta) n^{\frac{1}{2}}
  \int_{0}^1 (B_{\frac{kl+s}{n}}-B_{\frac{kl}{n}}) \dd \mu(s),
  \\
  \widetilde{U}_{j,l}&:=a(X_\frac{kl}{n},\theta) n^{\frac{1}{2}}
  \int_{0}^1 (B_{\frac{kl+j+s}{n}}-B_{\frac{kl+j-1+s}{n}}) \dd
  \mu(s)\quad \text{ for $j=1,\dots,k-1$,} \\
  \widetilde{U}_{k,l}&:=a(X_\frac{kl}{n},\theta) n^{\frac{1}{2}}
  \int_0^1 (B_{\frac{k(l+1)}{n}}-B_\frac{kl+k-1+s}{n} ) \dd \mu(s).
\end{align*}
Observe that this vector $(\widetilde{U}_{j,l})_{j=0,\dots,k}$ has,
under $\PP^\theta$ and conditionally to $X_\frac{kl}{n}=x_0$, the
same law as the vector $(\widetilde{U}_{j}^\theta)_{j=0,\dots,k}$
defined in Section \ref{S:A_Gaussian_approximation}.  Thus its
conditional law is Gaussian with covariance matrix
$a(X_{\frac{kl}{n}},\theta)^{2}\widetilde{K}$.  Hence, the variable
$\widetilde{\xi}_{l,n}(\theta_0)$ is $\F^n_{l+1}$--measurable and
under $\PP^{\theta_0}$, it is conditionally (to $X_\frac{kl}{n}$)
distributed as a recentered $\chi^2(k+1)$ variable.  Thus we deduce
the following four properties:
\begin{enumerate}
\item [1)] $u_n \sum_{l=0}^{L_n} E_{\theta_0}\left[
    \widetilde{\xi}_{l,n}(\theta_0) \mid \F_l^n \right] =0 $;
\item [2)] Using $u_n^2=1/n$, $L_n \sim n/k_n \to \infty$ and $k_n\to
  \infty$, one has
  \begin{align}\nonumber
    u^2_n \sum_{l=0}^{L_n} E_{\theta_0}\left[
      (\widetilde{\xi}_{l,n}(\theta_0))^2 \mid \F_l^n \right] & =u^2_n
    \sum_{l=0}^{L_n} 2(k_{n}+1)
    \left(\frac{\dot{a}}{a}\right)^2(X_\frac{kl}{n},\theta_0) \\
\label{E:termePalNnsur}
&\hspace{-1.5cm}=\frac{2(k_{n}+1) }{k_n} \int_0^1
\left(\frac{\dot{a}}{a}\right)^2(X_s,\theta_0) \dd s+
o_{\mathbb{P}^{\theta_0}}(1)
\\
\nonumber &\xrightarrow{\mathbb{P}^{\theta_0} }2 \int_0^1
\left(\frac{\dot{a}}{a}\right)^2(X_s,\theta_0) \dd s =
\mathcal{I}_{\theta_0};
  \end{align}
\item [3)] $ u^4_n \sum_{l=0}^{L_n} E_{\theta_0}\left[
    \abs{\widetilde{\xi}_{l,n}(\theta_0)}^4 \mid \F_l^n \right] \le c
  n^{-2} L_n k_n^4 \le c n^{-1} k_n^3 \to 0, $ if $k_n$ goes to
  $\infty$ slowly enough;
\item [4)] $ u_n \sum_{l=0}^{L_n} E_{\theta_0}\left[
    \widetilde{\xi}_{l,n}(\theta_0)
    [B_{\frac{(k+1)l}{n}}-B_{\frac{kl}{n}}] \mid \F_l^n \right] =0$.
\end{enumerate}
>From these four properties, it follows (see Jacod \cite{Jacod}) that $
u_n \sum_{l=0}^{L_n} \widetilde{\xi}_{l,n}(\theta_0) $ converges
stably under $\PP^{\theta_0}$ to a mixed Gaussian variable as in the
statement of the proposition. To obtain the limit for $N_n^\aug$, it
is sufficient to prove that
\begin{equation}
  \label{eq:manu3}
  N_n^\aug- u_n \sum_{l=0}^{L_n}
\widetilde{\xi}_{l,n}(\theta_0)\xrightarrow{\mathbb{P}^{\theta_0} }0.
\end{equation}
Due to Lemma \ref{lemma:9} a sufficient condition consists in the two
following points:
\begin{enumerate}
\item [1)] $u_n\sum_{l=0}^{L_n} E_{\theta_0}\left[
    \widetilde{\xi}_{l,n}(\theta_0) -\xi_{l,n}(\theta_0) \mid \F_l^n
  \right] \to 0$ in probability;
\item [2)] $u^2_n \sum_{l=0}^{L_n} E_{\theta_0}\left[
    (\widetilde{\xi}_{l,n}(\theta_0) -\xi_{l,n}(\theta_0))^2 \mid
    \F_l^n \right] \to 0$ in probability.
\end{enumerate}
But these two points can be shown using \eqref{E:diffUUtilde1} and
\eqref{E:diffUUtilde4} of Lemma \ref{L:compUUtilde} (for $k_n$ slowly
increasing).

\noindent
$\bullet$ \underline{We now study $I_n^\aug$.} A direct
differentiation of $\xi_{l,n}(\theta)$ (recall \eqref{E:def_xi}) gives
\begin{equation*}
\dot{\xi}_{l,n}(\theta)=\frac{\dot{a}}{a}(X_\frac{kl}{n},\theta_0)
\frac{-2\dot{a}}{a^3}(X_\frac{kl}{n},\theta) \sum_{0\le j,j' \le k}
U_{l,j} \widetilde{K}^{-1}_{j,j'} U_{l,j'}.
\end{equation*}
Then, with a few computations similar to the study of $N^\aug_n$, we
obtain (for appropriate $k_n$):
\begin{enumerate}
\item [1)] $\displaystyle u^2_n \sum_{l=0}^{L_n} E_{\theta_0}\left[
    \dot{\xi}_{l,n}(\theta_0) \mid \F_l^n \right]= u_n^2
  \sum_{l=0}^{L_n} -2 (k_{n,l}+1)
  \frac{\dot{a}^2}{a^2}(X_\frac{kl}{n},\theta_0)
  +O_{\mathbb{P}^{\theta_0}}(\frac{c(k_n)}{\sqrt{n}})$
\begin{align}
  \label{E:termePalInsur} & = -2 \frac{(k_{n}+1) }{k_n} \int_0^1
  \left(\frac{\dot{a}}{a}(X_s,\theta_0) \right)^2 \dd s+
  o_{\mathbb{P}^{\theta_0}}(1) \\ \nonumber &
  \xrightarrow{\mathbb{P}^{\theta_0}} - \mathcal{I}_{\theta_0};
\end{align}
\item [2)] $u^4_n \sum_{l=0}^{L_n} E_{\theta_0}\left[
    [\dot{\xi}_{l,n}(\theta_0)]^2 \mid \F_l^n \right] \le c n^{-1}
  k_n^4 \to 0$.
\end{enumerate}
Combined with Lemma \ref{lemma:9}, these two convergences imply that
of $I_n^\aug$ to $ \mathcal{I}_{\theta_0}$ under $\PP^{\theta_0}$.\\
$\bullet$ \underline{The remainder term $R_n$.}  Firstly
a direct use of \eqref{E:diffUUtilde1} gives $E(\abs{R^{(2)}_n})\le
c(k_n) n^{-\gamma/2} \to 0$ if $k_n$ slowly goes to $\infty$.
Secondly the convergence to zero of
$R^{(1)}_n=\sum_{l=0}^{L_n}R^{(1)}_{n,l}$ is more delicate and Lemma
\ref{lemma:9} is helpful for this. To this end we evaluate the
conditional expectation of $R^{(1)}_{n,l}$ using
\eqref{E:diffUUtilde4} and the fact the $(\widetilde{U}_{j,l})_j$
have the conditional covariance matrix $a(X_\frac{kl}{n},\theta)^2
\widetilde{K}$:
\begin{multline*}
  E_{\theta_0}[ R^{(1)}_{n,l} \mid \F_l^n]=
  \int_{\theta_0}^{\theta_0+u_n h}
  [\frac{\dot{a}}{a}(X_\frac{kl}{n},s)-\frac{\dot{a}}{a}(X_\frac{kl}{n},\theta_0)]
  \{ \frac{ a(X_\frac{kl}{n},\theta_0)^2}{a(X_\frac{kl}{n},s)^2} - 1
  \}(k_n+1) \dd s \\ + O(n^{-1}u_nc(k_n)).
\end{multline*}
The function $a$ being $\mathcal{C}^{1+\gamma}$ in $\theta$, one gets:
$ \sum_{l=0}^{L_n } \abs{E_{\theta_0}[ R^{(1)}_{n,l} \mid \F_l^n]} \le
c n^{-\gamma/2}+\frac{c(k_n)}{k_n}n^{-1/2}\to 0$ for appropriate
$k_n$. With similar considerations we evaluate the second conditional
moment and obtain $ u^2_n \sum_{l=0}^{L_n } E_{\theta_0}[
(R^{(1)}_n)^2 \mid \F_l^n] \le c(k_n)L_n u_n^{2+2\gamma}
\xrightarrow{n \to \infty} 0.$
\end{proof}

\subsection{Proof of Theorem \ref{T:LAMN_sur_observation}: the negligible terms}\label{S:The_negligible}
It remains to prove that, as announced, there is convergence to zero
of $\sum_{l=0}^{L_n} \eta_l$ with
$\eta_l=\int_{\theta_0}^{\theta_0+u_n h} r_{ X_{\frac{kl}{n}}
}(U_{0,l},\dots,U_{k,l},s) \dd s$. We aim at applying Lemma
\ref{lemma:9} by computing the first two conditional moments of
$\eta_l$ under $\PP^{\theta_0}$. The main difficulty here comes from
the fact that we do not have an explicit expression for
$r_{x_0}((u_j)_j,\theta)$. Indeed by Theorem
\ref{T:approx_score_gauss} we know bounds for the moments
$E^n_{\theta,x_0}\left( \abs{r_{x_0}((U_j)_j,\theta)}^p \right)$
where by $E_{x_0,\theta}^n$ we denote the expectation with respect
to the law of $\xr^{n,\theta}$ solution of \eqref{E:defXronde}. This
is a priori insufficient to compute the conditional moments of
$\eta_l$ under $\mathbb{P}^{\theta_0}$ which involve quantities such
as $E_{\theta_0,x_0}^n\left( \abs{r_{x_0}((U_j)_j,s)}^p \right)$ for
$s\neq \theta_0$. Thus in Lemmas
\ref{L:stabilite_taille}-\ref{L:stabilite_centrage} in the Appendix
we study the transformation of such moments under change of measure.

Firstly, we evaluate the conditional expectation of $\eta_{l}$,
$$E_{\theta_0}\left[ \eta_l \mid \F_l^n
\right]=\int_{\theta_0}^{\theta_0+u_n h}
E_{\theta_0,x}^n[r_x((U_j)_j,s)]_{\mid x=X_{\frac{kl}{n}}}\dd s.
$$
But $\abs{E^n_{\theta_0,x}[r_x((U_j)_j,s)]}\le
\abs{E^n_{s,x}[r_x((U_j)_j,s)]}+\abs{
  E^n_{\theta_0,x}[r_x((U_j)_j,s)]-E^n_{s,x}[r_x((U_j)_j,s)]}$ can be
bounded using \eqref{E:cond_r_centrage}
and Lemma \ref{L:stabilite_centrage} in the Appendix by $c(k)n^{-1}+
\abs{s-\theta_0} E^n_{s,x}[\abs{r_x((U_j)_j,
  s)}^\alpha]^{\frac{1}{\alpha}}$ for some $\alpha \ge 1$.
Then by \eqref{E:cond_r_moment} we deduce $\abs{E\left[ \eta_l \mid
    \F_l^n \right]} \le c(k_n) [ u_n n^{-1} + u_n^2 n^{-1/2} ]$. Finally a
block length $k_n$ slowly increasing guarantees $\sum_{l=0}^{L_n} E_{\theta_0}\left[ \eta_l \mid \F_l^n \right]\xrightarrow{ \mathbb{P}^{\theta_0}} 0.$

Secondly and similarly, owing to Theorem \ref{T:approx_score_gauss} and Lemma
\ref{L:stabilite_taille} in the Appendix, we get $ E\left[ \eta_l^2 \mid
  \F_l^n \right] \le c(k_n) u_n^2 n^{-1}\to 0$. Therefore, by Lemma \ref{lemma:9}, we have proved $\sum_{l=0}^{L_n} \eta_l\xrightarrow{\mathbb{P}^{\theta_0}}0$.  This ends the proof of Theorem \ref{T:LAMN_sur_observation}.

\subsection{Proof of Theorem \ref{T:LAMN_k_fixe}}
%
  The proof is essentially the same as that of Theorem
  \ref{T:LAMN_sur_observation}, the difference in the asymptotic
  information comes from the difference in the limit of the quantities
  \eqref{E:termePalNnsur} and \eqref{E:termePalInsur} when $k$ is fixed.
%
\section{LAMN property for the initial model}\label{S:LAMN_property}
In this Section we are back to the model where the observation is only
$\mathcal{O}^n=(\Xb_j)_{j=0,\dots,n-1}$ and we will prove Theorem
\ref{T:LAMN_observation} by relying on the LAMN property for the
augmented model.

A first intermediate result is that one can approximate the
log-likelihood of the augmented model by a function of the observation
$\mathcal{O}^n$.

\begin{prop}\label{P:pseudo_vraisem}
  There exist random variables $\Gamma_n$ measurable with respect to
  $\mathcal{O}^n$ such that:
  $$
  \ln (Z^{n,\aug}_{\theta_0,\theta_0+u_nh}(\mathcal{O}^{n,\aug}))-
  \Gamma_n \xrightarrow[\mathbb{P}^{\theta_0}]{n \to \infty} 0.
  $$
\end{prop}
\begin{proof}
  We have seen in Section \ref{S:Asymptotic_study} that $\ln
  (Z^{n,\aug}_{\theta_0,\theta_0+u_nh}(\mathcal{O}^{n,\aug}))=
  hN_n^\aug-1/2h^2I_n^\aug+o_{\mathbb{P}^{\theta_0}}(1)$ where the
  quantities $N_n^\aug$ and $I_n^\aug$ were defined in Section
  \ref{S:The_explicit}.

  Thus the proof of the proposition consists in introducing a proper
  modification of these quantities which only depends on the
  observations. We let for $l=0,\dots,k_n$
\begin{equation*}\label{E:defxiObs}
\xi_{l,n}^\obs(\theta)=
\frac{\dot{a}}{a}(\Xb_{kl-1},\theta_0) \left\{
a^{-2}(\Xb_{kl-1},\theta)
\sum_{1 \le j ,j' \le k-1} U_{l,j} \widehat{K}^{-1}_{j,j'} U_{l,j'} - (k-1)
\right\},
\end{equation*}
with the convention $\Xb_{-1}=\xi_0$ is the known initial value of the
diffusion and the matrix $a^2(x_0,\theta)\widehat{K}$ is the
covariance matrix of the conditionally Gaussian vector
$(\widetilde{U}^\theta_1,\quad,\widetilde{U}^\theta_{k-1})$:
\begin{equation*}
\widehat{K}=
\begin{bmatrix}
  v_1+v_2 & c    &0  & 0   \\
  c   & \ddots  & \ddots &0 \\
  0 & \ddots &\ddots &c \\
  0 &0  & c & v_1+v_2 \\
\end{bmatrix}.
\end{equation*}
Clearly, $\xi_{l,n}^\obs(\theta)$ only depends on the observation
$\mathcal{O}^n$ since we have suppressed all occurrences of the
variables $U_{0,l}$ and $U_{k,l}$ and replaced $X_{\frac{kl}{n}}$ by
$\Xb_{kl-1}$ in the expression of $\xi_{l,n}(\theta)$ (compare with
\eqref{E:def_xi}).  Then we let $N^\obs_n=u_n\sum_{l=0}^{L_n}
\xi_{l,n}^\obs(\theta_0)$ and $I^\obs_n=-u^2_n\sum_{l=0}^{L_n}
\frac{\partial \xi_{l,n}^\obs}{\partial \theta}(\theta_0)$.\\
$\bullet$ \underline{Study of $N^\aug_n-N^\obs_n$.}  The first step is
to consider the conditionally recentered chi square approximation of
$\xi_{l,n}^\obs(\theta)$ that we define as:
\begin{equation}\label{E:defxitildeobs}
\hspace{-1mm}\widetilde{\xi}_{l,n}^\obs(\theta)=\frac{\dot{a}}{a}(X_{\frac{kl}{n}},\theta_0)
\left\{
a^{-2}(X_{\frac{kl}{n}},\theta)
\sum_{1 \le j ,j' \le k-1}
\widetilde{U}_{l,j}\widehat{K}^{-1}_{j,j'}
\widetilde{U}_{l,j'} - (k-1)
\right\}.
\end{equation}
The first step is to prove the validity of the approximation:
\begin{equation} \label{E:diffXiXitilde}
u_n \sum_{l=0}^{L_n}
\{\xi^\obs_{l,n}(\theta_0)-\widetilde{\xi}^\obs_{l,n}(\theta_0)\} \xrightarrow{\mathbb{P}^{\theta_0} }0.
\end{equation}
This is done similarly to the proof of
$N^\aug_n-u_n\sum_{l=0}^{L_n}\widetilde{\xi}_{l,n}(\theta_0) \to 0$ in
proposition \ref{P:termPal}, by considering the first two conditional
moments, but here the first moment is more delicate to handle: the
conditional moment $E_{\theta_0}[
\xi^\obs_{l,n}(\theta_0)-\widetilde{\xi}^\obs_{l,n}(\theta_0) \mid
\F^n_l]$ is of the form $(k-1)\{ g( X_{\frac{kl}{n}} )-g( \Xb_{kl-1})
\} h(\Xb_{kl-1})+O(c(k_n)/n)$ for $g$ and $h$ two $\mathcal{C}^2$
functions.
If we abruptly use the relation $\norme{X_{\frac{kl}{n}}
  -\Xb_{kl-1}}_{\mathbf{L}^p} \le c(p)n^{-1/2}$ then we only deduce
that $u_n \sum_{l=0}^{L_n} E_{\theta_0}[
\xi^\obs_{l,n}(\theta_0)-\widetilde{\xi}^\obs_{l,n}(\theta_0) \mid
\F^n_l]$ remains bounded in probability. To show that it actually
converges to zero, we have to apply again Lemma \ref{lemma:9} to the
new triangular array of variables, $u_n \sum_{l=0}^{L_n} (k-1)\{ g(
X_{\frac{kl}{n}} )-g( \Xb_{kl-1}) \} h(\Xb_{kl-1})$.  Then by rather
long computations, using that
$\norme{\Xb_{kl-1}-\Xb_{kl-2}}_{\mathbf{L}^p} \le c(p) (k/n)^{1/2}$
and $\abs{E_{\theta_0}[X_\frac{kl}{n} -\Xb_{kl-1} \mid \F^n_{l-1}]}
\le c n^{-1}$,
we can prove,
\begin{align*}
  u_n \sum_{l=0}^{L_n} (k-1) \abs{E_{\theta_0}[ \{ g( X_{\frac{kl}{n}}
    )-g( \Xb_{kl-1}) \} h(\Xb_{kl-1}) \mid \F^n_{l-1}]} \le c(k)
  n^{-1/2} \xrightarrow{\mathbb{P}^{\theta_0} }0,
  \\
  u_n^2 \sum_{l=0}^{L_n} (k-1)^2 E_{\theta_0}[ \{ g( X_{\frac{kl}{n}}
  )-g( \Xb_{kl-1}) \}^2 h(\Xb_{kl-1})^2 \mid \F^n_{l-1}] \le c(k)
  n^{-1} \xrightarrow{\mathbb{P}^{\theta_0} }0.
\end{align*}
Thus we deduce $u_n \sum_{l=0}^{L_n} E_{\theta_0}[
\xi^\obs_{l,n}(\theta_0)-\widetilde{\xi}^\obs_{l,n}(\theta_0) \mid
\F^n_l] \to 0$.  The second condition $u_n^2 \sum_{l=0}^{L_n}
E_{\theta_0}[
(\xi^\obs_{l,n}(\theta_0)-\widetilde{\xi}^\obs_{l,n}(\theta_0))^2 \mid
\F^n_l] \le c(k)n^{-1} \to 0$ is easily obtained and we deduce
\eqref{E:diffXiXitilde}.

Thus, in view of the equation \eqref{eq:manu3}, it remains to prove that $u_n \sum_{l=0}^{L_n} \{
  \widetilde{\xi}^\obs_{l,n}(\theta_0)-\widetilde{\xi}_{l,n}(\theta_0)
\}$ is negligible. But by Lemma \ref{L:diff_chi_deux} in the
Appendix, comparing expressions \eqref{E:defxitilde} and
\eqref{E:defxitildeobs}, it appears that conditionally to $\F^n_l$ the
random variable
$\widetilde{\xi}^\obs_{l,n}(\theta_0)-\widetilde{\xi}_{l,n}(\theta_0)$
is a recentered $\chi^2(2)$ variable and hence the following
properties hold:
\begin{align*}
  u_n \sum_{l=0}^{L_n} E_{\theta_0}\left(
    \widetilde{\xi}^\obs_{l,n}(\theta_0)-\widetilde{\xi}_{l,n}(\theta_0)
    \mid \F^n_l \right)&=0,
  \\
  u_n^2 \sum_{l=0}^{L_n} E_{\theta_0}\left( \left\{ \widetilde{\xi}^\obs_{l,n}
      (\theta_0)-\widetilde{\xi}_{l,n}(\theta_0) \right\}^2 \mid
    \F^n_l \right)&=\sum_{l=0}^{L_n} u_n^2 4
  \frac{\dot{a}^2}{a^2}(X_{\frac{kl}{n}},\theta_0) \le \frac{c}{k_n}
  \to 0.
\end{align*}
These two properties imply by Lemma \ref{lemma:9} the convergence to 0 under $\PP^{\theta_0}$ of $u_n \sum_{l=0}^{L_n} \left\{
  \widetilde{\xi}^\obs_{l,n}(\theta_0)-\widetilde{\xi}_{l,n}(\theta_0)
\right\}$, and thus $N^\aug_n-N^\obs_n\xrightarrow{\mathbb{P}^{\theta_0} }0$.\\
$\bullet$ \underline{Study of $I^\aug_n-I^\obs_n$.} Exactly as we
proved that $I^\aug_n$ tends to $\mathcal{I}_{\theta_0}$ we can show
that $I^\obs_n \to \mathcal{I}_{\theta_0}$. Thus the difference is
negligible.

Finally the proposition is obtained by setting $\Gamma_n =h
N^\obs_n- h^2/2 I^\obs_n$.
\end{proof}
Then  Theorem \ref{T:LAMN_observation} is a consequence of the
following proposition combined with Proposition
\ref{P:pseudo_vraisem} and Theorem \ref{T:LAMN_sur_observation}.
\begin{prop}
  We have the convergence,
  $$
  Z^n_{\theta_0,\theta_0+u_n h} - e^{\Gamma_n}
  \xrightarrow[\mathbb{P}^{\theta_0}]{n \to \infty}0.
  $$
\end{prop}
\begin{proof}
  The starting point is the relation between the likelihood of the
  initial and of the augmented model: $Z^n_{\theta_0,\theta_0+u_n
    h}=E_{\theta_0}\left[ Z^{n,\aug}_{\theta_0,\theta_0+u_n h} \mid
    \mathcal{O}^n \right]$.  By Proposition \ref{P:pseudo_vraisem} we
  can write $Z^{n,\aug}_{\theta_0,\theta_0+u_n
    h}=e^{\Gamma_n}e^{\ve_n}$ where $\ve_n$ tends to zero in
  $\mathbb{P}_{\theta_0}$ probability. Using that $\Gamma_n$ is
  $\mathcal{O}^n$ measurable we deduce,
  $$
  Z^n_{\theta_0,\theta_0+u_n h}-e^{\Gamma_n}=
  E_{\theta_0}\left[e^{\Gamma_n} ( e^{\ve_n} -1) \mid \mathcal{O}^n
  \right].
  $$
  We now use the inequality $\abs{e^u-1}\le (\abs{u} \wedge 1)
  (e^u+1)$ to obtain that $\abs{Z^n_{\theta_0,\theta_0+u_n
      h}-e^{\Gamma_n}} \le \alpha_n+\beta_n$ with:
\begin{align*}
  \alpha_n&=E_{\theta_0}\left[(\abs{\ve_n}\wedge 1) e^{\Gamma_n} \mid
    \mathcal{O}^n \right] = E_{\theta_0}\left[ \abs{\ve_n}\wedge 1
    \mid \mathcal{O}^n \right] e^{\Gamma_n},
  \\
  \beta_n&=E_{\theta_0}\left[(\abs{\ve_n}\wedge 1) e^{\Gamma_n}
    e^{\ve_n} \mid \mathcal{O}^n \right] = E_{\theta_0 }\left[
    (\abs{\ve_n}\wedge 1) Z^{n,\aug}_{\theta_0,\theta_0+u_n h} \mid
    \mathcal{O}^n \right].
\end{align*}
It now remains to show the convergence to zero of $\alpha_n$ and
$\beta_n$.

For $\alpha_n$, let us notice that $(e^{\Gamma_n})_n$ is a tight
sequence and that $E_{\theta_0}\left[\abs{\ve_n}\wedge 1 \mid
  \mathcal{O}^n \right]$ converges in
$\mathbf{L}^1(\mathbb{P}_{\theta_0})$ norm to zero since,
$$
E_{\theta_0}\left[ E_{\theta_0}\left[\abs{\ve_n}\wedge 1 \mid
    \mathcal{O}^n \right] \right] =E_{\theta_0} [\abs{\ve_n}\wedge 1]
\xrightarrow{n \to \infty} 0.
$$

For $\beta_n$, we have
$E_{\theta_0}[\beta_n]=E_{\theta_0+u_nh}[\abs{\ve_n}\wedge 1]$.  But
the sequence of probabilities $\mathbb{P}^{\theta_0}$ and
$\mathbb{P}^{\theta_0+u_nh}$ restricted to the sigma fields
$\mathcal{O}^{n,\aug}$ are contiguous (this is a consequence of the
LAMN property for the augmented model, see e.g. Proposition 1 in
Jeganathan \cite{Jeganathan1}); hence the sequence $(\ve_n)_n$ which
is measurable with respect to $\mathcal{O}^{n,\aug}$ and converges
to zero in $\mathbb{P}^{\theta_0}$--probability converges also in
$\mathbb{P}^{\theta_0+u_nh}$--probability. This implies $
E_{\theta_0}[\beta_n]= E_{\theta_0+u_nh}[\abs{\ve_n}\wedge 1] \to
0$.
\end{proof}

\section{Appendix}
\subsection{Proof of results of Section \ref{S:The_density}}\label{S:Proof_of}
Since the results of Section \ref{S:The_density} concern only the
study of a density for fixed values of $\theta$, we omit the
dependence upon $\theta$ in our notations. We will prove the results
in the following order.  First in section \ref{Ss:Existence_of}, we
show that the law of the Wiener functional $(U^n,V^n)=\left( \int_0^1
  \xr_s^{n} \dd \mu(s),\xr_1^{n} \right)$
admits a density.  Then we prove the lower and upper bounds given in
Theorem \ref{T:densiteUV} (section \ref{Ss:Lower_bound}) and
eventually we deduce the Proposition \ref{P:inversibleMalliaB}
(section \ref{S:Proof_of_Proposition}).

\subsubsection{Existence of the density $\pUV^n_{x_0}$}
\label{Ss:Existence_of}
We know \cite{Nualart1} that under \ref{regulariteR} the random
variable $\xr^n_t$ is an element of $\D^{3,\infty}$ and its first
derivative is equal to
\begin{equation}\label{E:formulDX}
D_t \xr^{n}_s=\ind{t \le s} \yr^n_s (\yr^n_t)^{-1}a_n(\xr^n_t),
\end{equation}
where $\yr^n$ is the solution of
\begin{equation}
\dd \yr^n_t=a'_n(\xr^n_t) \yr^n_t \dd W_t + b_n'(\xr^n_t) \yr^n_t \dd t,\quad \yr^n_0=1.
\end{equation}
In the sequel we will repeatedly use the positivity of $\yr^n$ and the control
\begin{equation}\label{E:borneY}
E(\sup_{t \in [0,1]} (\yr^n_t)^p)+
E(\sup_{t \in [0,1]} (\yr^n_t)^{-p}) \le c(p).
\end{equation}
>From this we can see that the random variables $U^n$ and $V^n$ are
elements of $\D^{3,\infty}$ and using \eqref{E:formulDX} with the
linearity of the operator $D$, we have
\begin{align*}
  D_tU^n&=\int_0^{1} \ind{t \le s } \yr^n_s (\yr^n_t)^{-1}
  a_n(\xr^n_t)\dd \mu(s)= a_n(\xr^n_t) (\yr^n_t)^{-1} \ind{t \le 1} \int_{[t,1]} \yr^n_s
  \dd \mu(s),
  \\
  D_t V^n&=a_n(\xr^n_t) \yr_1^{n} (\yr_t^n)^{-1}\ind{t \le 1}.
\end{align*}
Using Theorem 2.1.2 p. 86 in \cite{Nualart1}, a sufficient condition
for the existence of a density for $(U^n,V^n)$ is that its Malliavin
covariance matrix $\gamma_{U^n,V^n}$ satisfies a non degeneracy
condition given, for instance, by the following lemma.
\begin{lem}\label{L:mino_mallia_taille2}
  $\gamma_{U^n,V^n}$ is an $a.s$ invertible matrix and for all $p\geq
  1$, we have
  $$
  E\left( \abs{\det(\gamma_{U^n,V^n})}^{-p} \right) \le c(p).
  $$
\end{lem}
\begin{proof}
  To have shorter notations, during the proof we will denote by $c_*$
  any generic positive random variable which satisfies $E(c_*^{-p})
  \le c(p)$.
  By direct computations we have,
\begin{align}\label{E:croUU}
  \psH{U^n}{U^n}&= \int_0^{1} a_n^2(\xr^n_t)(\yr^n_t)^{-2} \left(
    \int_{[t,1]} \yr^n_s \dd \mu(s) \right)^2 \dd t, \\\label{E:croUV}
  \psH{U^n}{V^n}&= \int_0^{1} a_n^2(\xr^n_t)(\yr^n_t)^{-2} \left(
    \int_{[t,1]} \yr^n_s \dd \mu(s) \right) \dd t \: \yr^n_1,
  \\\label{E:croVV} \psH{V^n}{V^n}&= \int_0^{1}
  a_n^2(\xr^n_t)(\yr^n_t)^{-2} \dd t \: (\yr^n_{1})^2.
\end{align}
Now, define the probability density on $[0,1]$
\begin{equation}\label{E:def_m}
m^n_t= a_n^2(\xr^n_t)(\yr^n_t)^{-2}
\left(\int_0^1 a_n^2(\xr^n_s)(\yr^n_s)^{-2} \dd s \right)^{-1}
,
\end{equation}
and set $f^n(t):=\int_{[t,1]} \yr^n_s \dd \mu(s)$. Thus we can
write:
\begin{equation*}
\det(\gamma_{U^n,V^n}) = \psH{V^n}{V^n}^2 (\yr_1^n)^{-2}
\left[ \int_0^{1} m^n_t f^n(t)^2 \dd t
-
\left( \int_0^{1} m^n_s  f^n(s) \dd s \right)^2 \right].
\end{equation*}
Hence the above bracket can be interpreted as the variance of the function
$f^n(t)$ under the probability measure $m^n_t \dd t$ and hence:
\begin{equation*}
\det(\gamma_{U^n,V^n}) = \psH{V^n}{V^n}^2 (\yr^n_1)^{-2}
\int_0^{1} m^n_t \left[
f^n(t)-
\left(
\int_0^{1} m^n_r  f^n(r) \dd r
\right)\right]^2 \dd t.
\end{equation*}
But clearly under Assumption \ref{regulariteRprime}, $\psH{V^n}{V^n}^2
\ge \underline{a}^2 \inf_{t \in [0,1]} (\yr^n_t)^{-2} \inf_{t \in
  [0,1]} (\yr^n_t)^{2}$ and hence by \eqref{E:borneY} this yields, $
\psH{V^n}{V^n}^2 \ge c_*, $ using our convention about
generic positive random variables $c_*$.  Similarly, by \eqref{E:def_m}, we have
$m^n_t \ge c_*$ and thus,
$$
\det(\gamma_{U^n,V^n}) \ge c_* \int_0^{1} \left[ f^n(t)- \left(
    \int_0^1 m^n_r f^n(r) \dd r \right)\right]^2 \dd t.
$$
Then, writing the integral above as
$$
\int_0^{1/2} \left[ f^n(t)- \left( \int_0^{1} m^n_r f^n(r) \dd r
  \right)\right]^2 + \left[ f^n(t+1/2)- \left( \int_0^1 m^n_r f^n(r)
    \dd r \right)\right]^2 \dd t,$$
and using the simple inequality
$x^2+y^2\ge (x-y)^2 /2$, we get: $\det(\gamma_{U^n,V^n}) \ge c_* \int_0^{1/2} \left( \int_{[t,t+\frac{1}{2})}
    \yr^n_s \dd \mu(s) \right)^2 \dd t .$
Using again $\inf_{s \in [0,1]} \yr^n_s \ge c_*$, we obtain:\break
$ \det(\gamma_{U^n,V^n}) \ge c_* \int_0^{1/2} \mu\left( [t,t+1/2)
\right)^2 \dd t.  $
But this integral is positive as soon as $\mu\left((0,1)\right)>0$ which is the case by assumption \eqref{E:hypomu}. Thus the lemma is proved.
\end{proof}

\subsubsection{Bounds for the density}\label{Ss:Lower_bound}
For the proof of \eqref{E:maj_min_densite}, we make a crucial use of
the fact that the diffusion process $\xr^n$ is one dimensional by
introducing the classical transformation:
$$
s_n(x):=\int_0^x a_n^{-1}(y) \dd y, \quad \W^n_t:=s_n(\xr^n_t).
$$
By the assumptions on $a$, the function $s_n$ is one to one on $\R$
and the derivatives of $s_n$ and $s_n^{-1}$ are bounded
independently of $n$.  By It\^o's formula, $\W^n$ solves the equation $\dd
\W^n_t=\dd W_t + \tilde{b}_n(\W^n_t) \dd t$ where
$\tilde{b}_n(w):=\frac{b_n}{a_n}\circ s^{-1}_n(w)-\frac{1}{2}
a'_n\circ s^{-1}_n(w)$ and the initial value is
$\W^n_0=s_n(\xr^n_0)=0$.  We let $\tilde{P}$ be the probability
defined on $(\Omega,\mathcal{A})$ by
\begin{equation*}
\frac{\dd \tilde{P}}{\dd P}=\exp
\left(
-\int_0^1 \tilde{b}_n(\W^n_u) \dd W_u -
\frac{1}{2}\int_0^1
\tilde{b}_n^2(\W^n_u) \dd u
\right).
\end{equation*}
The Girsanov theorem implies that the process $\W^n$ is under
$\tilde{P}$ a standard Brownian motion. Note that the random
variables $(U^n,V^n)$ have the following expressions with respect to
this $\tilde{P}$--Brownian motion:
\begin{align}\label{E:def_Ungauss}
  U^n&=\int_0^{1} s^{-1}_n (\W^n_r) \dd \mu(r), \\
  \label{E:def_Vngauss} V^n&=s^{-1}_n(\W^n_1).
\end{align}
Now let $h_0$ and $h_1$ be non negative real functions, then:
\begin{equation}\label{E:girsanovPPt}
E_P \left[ h_0(U^n)h_1(V^n)\right]=
E_{\tilde{P}} \left[ h_0(U^n)h_1(V^n) L^n \right],
\end{equation}
where $L^n=\exp \left( \int_0^1 \tilde{b}_n(\W^n_r) \dd \W^n_r
  -\frac{1}{2} \int_0^1 \tilde{b}_n^2(\W^n_r) \dd r \right)$.  But
using It\^o's formula, $L^n=\exp \left( \tilde{B}_n(\W_1^n)
  -\frac{1}{2}\int_0^1 (\tilde{b}_n^2+\tilde{b}_n')(\W^n_r) \dd r
\right)$ where $\tilde{B}_n$ is the primitive function of
$\tilde{b}_n$ vanishing at zero.  Since $\tilde{b}_n$ and
$\tilde{b}_n'$ are clearly bounded by $cn^{-1/2}$ for some constant
$c$ only depending on $a$ and $b$ and $\abs{\tilde{B}_n(x)}\le c
n^{-1/2} \abs{x}$ we have: $ c^{-1} \exp\left( - c
  n^{-1/2}\abs{\W^n_{1}} \right) \le L^{n} \le c \exp\left( c n^{-1/2}
  \abs{\W^n_{1}} \right)$.  By \eqref{E:def_Vngauss} and the boundedness
of $s'_n$ we deduce $ c^{-1} \exp\left( - c n^{-1/2}\abs{V^n} \right)
\le L^{n} \le c \exp\left( c n^{-1/2} \abs{V^n} \right) $.  From this
and \eqref{E:girsanovPPt}, we obtain:
\begin{multline*}
  c^{-1} E_{\tilde{P}} \left[ h_0(U^n)h_1(V^n) e^{-c n^{-\frac{1}{2}}
      \abs{V^n}}\right]
  \\
  \le E_P \left[ h_0(U^n)h_1(V^n)\right] \le c E_{\tilde{P}} \left[
    h_0(U^n)h_1(V^n) e^{c n^{-\frac{1}{2}} \abs{V^n}}\right].
\end{multline*}
Hence we have transformed the problem of finding bounds for the
density of the law of $(U^n,V^n)$ under $P$ into an analogous problem
under $\tilde{P}$.  Consequently the bounds for $\pUV_{x_0}^n$ stated
in \eqref{E:maj_min_densite} will follow from the next lemma.
\begin{lem}
  Let $h_0$, $h_1$ be two non negative functions. There exist some
  constants $c_1>c_2>0$, depending only on the coefficients $a$ and
  $b$ such that:
\begin{multline*}
  c_1^{-1} \int \int h_0(u)h_1(v) e^{-c_1(u^2+v^2)} \dd u \dd v
  \le \\
  E_{\tilde{P}} \left[ h_0(U^n)h_1(V^n) \right] \le c_2^{-1} \int \int
  h_0(u)h_1(v) e^{-c_2(u^2+v^2)} \dd u \dd v.
\end{multline*}
\end{lem}
\begin{proof}
  We first show the lower bound.  Using that the random variable $V^n$
  is measurable with respect to $\W_1^n$ (by \eqref{E:def_Vngauss}),
  we can write:
\begin{align}\nonumber
  E_{\tilde{P}} \left[ h_0(U^n)h_1(V^n) \right] &= E_{\tilde{P}}
  \left[h_1(V^n) E_{\tilde{P}} \left[ h_0(U^n) \mid \W_1^n
    \right]\right] \\ \label{E:eq_gfrak} & = \int \mathfrak{g}(w)
  h_1\left(s^{-1}_n(w)\right) E_{\tilde{P}} \left[ h_0(U^n) \mid
    \W_1^n=w \right] \dd w,
\end{align}
where $\mathfrak{g}$ is the density of the standard Gaussian law.  Now
let us admit temporarily the following relation on the conditional law
of $U^n$:
\begin{equation}\label{E:admitcondUn}
E_{\tilde{P}}\left[ h_0(U^n) \mid \W_1^n\right] \ge c^{-1}
e^{-c(\W_1^n)^2}
\int h_0(u)e^{-c u^2 }\dd u.
\end{equation}
Then $ E_{\tilde{P}} \left[ h_0(U^n)h_1(V^n) \right]$ is greater than:
$$c^{-1}\int h_0(u)e^{-c u^2 }\dd u \times\int
\mathfrak{g}(w)h_1\left(s_n^{-1}(w)\right)e^{-cw^2}\dd w.$$
The change
of variable $v=s_n^{-1}(w)$ in the second integral above, the
inequalities $\abs{w} \le c \abs{v}$ and $s_n'(v) \ge c$ give the new
lower bound
$$c^{-1}\int h_0(u)e^{-c u^2 }\dd u \times\int
\mathfrak{g}\left(s_n(v)\right)h_1\left(v\right)e^{-cv^2}\dd v,$$
with
a new constant $c$. Since $\mathfrak{g}$ is the Gaussian kernel and
thanks to the inequality $\abs{s_n(v)} \le c \abs{v}$, we deduce the
required lower bound for $E_{\tilde{P}} \left[ h_0(U^n)h_1(V^n)
\right]$.

We obtain the upper bound quite similarly. Let us temporarily admit
that for all $\varepsilon$ small enough there exists $c(\varepsilon)$
such that:
\begin{equation}\label{E:admitcondUndeux}
E_{\tilde{P}}\left[ h_0(U^n) \mid \W_1^n\right] \le c(\varepsilon)^{-1}
e^{\varepsilon(\W_1^n)^2}
\int h_0(u)e^{-c(\varepsilon) u^2 }\dd u.
\end{equation}
Plugging this in equation \eqref{E:eq_gfrak}, we deduce that
$E_{\tilde{P}}\left[ h_0(U^n) \mid \W_1^n\right]$ is smaller than
$$c(\varepsilon)^{-1}\int h_0(u)e^{-c(\varepsilon) u^2 }\dd u
\times\int \mathfrak{g}(w)e^{\varepsilon
  w^2}h_1\left(s_n^{-1}(w)\right)\dd w.$$
Since
$\mathfrak{g}(w)=\exp(-w^2/2)/\sqrt{2\pi}$, any choice of
$\varepsilon$ smaller than $1/4$ implies that the second integral in
the equation above is bounded by $c \int e^{-\frac{1}{4} w^2}
h_1\left(s_n^{-1}(w)\right)\dd w$. As for the lower bound, we conclude
by the change of variable $v=s_n^{-1}(w)$.\end{proof}

It remains to show \eqref{E:admitcondUn}--\eqref{E:admitcondUndeux}. This is
done in the following lemma.
\begin{lem} For some constant $c>0$ and $\overline{\varepsilon}>0$, we have
\begin{equation}\label{E:borne_cond_min}
E_{\tilde{P}}\left[ h_0(U^n) \mid \W_1^n \right] \ge c^{-1}
e^{-c(\W_1^n)^2}
\int h_0(u)e^{-c u^2 }\dd u.
\end{equation}
For all $\varepsilon \in )0,\overline{\varepsilon}($, there exists
$c(\varepsilon)>0$ such that,
\begin{equation}\label{E:borne_cond_maj}
E_{\tilde{P}}\left[ h_0(U^n) \mid \W_1^n \right] \le c(\varepsilon)^{-1}
e^{\varepsilon (\W_1^n)^2}\int h_0(u)e^{-c(\varepsilon) u^2 }\dd u.
\end{equation}
\end{lem}
\begin{proof}
  Let us recall that the process $\W_t^*:=\W^n_t-t \W_1^n $ is a
  Brownian bridge on $[0,1]$, independent of the variable $\W_1^n$.
  Thus, we can evaluate the conditional expectation
  $E_{\tilde{P}}\left[ h_0(U^n) \mid \W_1^n=w \right] $ as the
  expectation (recall \eqref{E:def_Ungauss}),
\begin{equation}\label{E:ecritPont}
E\left[ h_0\left(
\int_0^1
s_n^{-1}\left( \W_t^*+tw \right) \dd \mu(t)
\right)\right],
\end{equation}
for $\W^*$ some Brownian bridge.  This Brownian bridge itself admits a
decomposition
\begin{equation}\label{E:decompPont}
\W_t^*=\xi \eta_t+\W_t^{**},
\end{equation}
where $\xi$ is a $\mathcal{N}(0,1)$ variable, $\eta$ is the
deterministic triangle shaped function:
$$
\eta_t=\begin{cases} t &\text{if $t\in [0,1/2]$}
  \\
  (1-t) & \text{if $ t\in [1/2,1]$}
\end{cases},
$$
and $\W^{**}$ is the process on $[0,1]$ constructed as the
concatenation of two independent Brownian bridges, one on $[0,1/2]$
and another on $[1/2,1]$.  Furthermore in this decomposition
the {r.v.\ $\eta$} and the process $\W^{**}$ are independent.

For any realization of $\W^{**}$ we can introduce the real function,
$$
x \mapsto g_{\W^{**}}(x)= \int_0^{1} \left\{ s^{-1}_n \left(
    x\eta_t+\W_t^{**}+tw \right) \right\} \dd \mu(t).
$$
Using \eqref{E:decompPont} and the independence of $\xi$ and
$\W^{**}$, the quantity \eqref{E:ecritPont} now writes,
\begin{equation}  \label{E:intInOut}
E_{(\W^{**})}  E_{(\xi)} \left[h_0(g_{\W^{**}}(\xi) )\right]
\end{equation}
where the inner expectation denotes the expectation with respect to
the random variable $\xi$ and the outer one with respect to the
process $\W^{**}$.

First we evaluate the inner expectation. Using that $\xi$ is a
standard Gaussian variable we have
\begin{equation} \label{E:expect_int}
 E_{(\xi)} \left[h_0(g_{\W^{**}}(\xi) )\right]=(2\pi)^{-1/2}
 \int h_0(g_{\W^{**}}(x)) e^{-\frac{x^2}{2}} \dd x.
\end{equation}
Note now that for any realization of $\W^{**}$, the function $x\mapsto
g_{\W^{**}}(x)$ is differentiable and using that $\frac{1}{c}\le
(s^{-1}_n)' \le c$ we get
\begin{equation*}\label{E:borngWprime}
\frac{1}{c}
\int_0^{1}
\eta_t    \dd \mu(t)
\le
 g_{\W^{**}}'(x) \le c
\int_0^{1}
\eta_t    \dd \mu(t).
\end{equation*}
By assumption \eqref{E:hypomu} on the measure $\mu$ the integral
$\int_0^1 \eta_t \dd\mu(t)$ is positive. Thus the function $x\mapsto
g_{\W^{**}}(x)$ is invertible on $\mathbb{R}$, with a derivative
bounded from above and from below by some constant independent of
$\W^{**}$ and $n$.  This allows us to make a change of variable in
\eqref{E:expect_int} to obtain the bounds
\begin{equation}\label{E:double_enca}
c^{-1}
\int h_0(u) e^{- \frac{\left( g_{\W^{**}}^{-1}(u) \right)^2 }{2}} \dd u \le
E_{(\xi)} \left[h_0(g_{\W^{**}}(\xi) )\right]\le c
 \int h_0(u) e^{- \frac{\left( g_{\W^{**}}^{-1}(u) \right)^2 }{2}} \dd u.
\end{equation}

Now the proofs of \eqref{E:borne_cond_min} and
\eqref{E:borne_cond_maj} are
treated separately. \\
$\bullet$ For the lower bound, we have seen that $g_{\W^{**}}^{-1}$
is globally Lipschitz with a constant independent of $\W^{**}$ and
thus $\abs{g_{\W^{**}}^{-1}(u)} \le c \abs{u} +
\abs{g_{\W^{**}}^{-1}(0)} \le c \abs{u} + c \abs{g_{\W^{**}}(0)}$.
In addition, a simple computation from the definition of
$g_{\W^{**}}$ and then boundedness of $(s_n^{-1})'$ show that
\begin{equation} \label{E:majgWstar}
 \abs{g_{\W^{**}}(0)} \le c [\abs{w}+ \sup_{t\in[0,1]}
 \abs{\W^{**}_t}].
\end{equation}
Using this in \eqref{E:double_enca} we find a new lower bound for the
inner expectation:
$$
E_{(\xi)} \left[h_0(g_{\W^{**}}(\xi) )\right]\ge c^{-1} e^{-c w^2}
e^{-c \sup_{t\in[0,1]} (W^{**}_t)^2} \int h_0(u) e^{-c u^2} \dd u .
$$
Taking the expectation with respect to $\W^{**}$ proves that
\eqref{E:intInOut} is larger than $c^{-1}  E_{(\W^{**})} \left( e^{-c \sup_{t\in[0,1]} (W^{**}_t)^2} \right) \times e^{-c w^2} \int h_0(u) e^{-c u^2} \dd u.$ This gives \eqref{E:borne_cond_min}.\\
$\bullet$ For the upper bound, we write using that $g_{\W^{**}}$ is
Lipschitz $\abs{ \abs{u} - \abs{g_{\W^{**}}(0)} } \le \abs{
  g_{\W^{**}}( g_{\W^{**}}^{-1}(u)) - g_{\W^{**}}(0) } \le
c\abs{g_{\W^{**}}^{-1}(u)}$. Together with the relation $(x-y)^2 \ge
x^2 \frac{\varepsilon}{1+\varepsilon}-\varepsilon y^2$ (for $x,y \in
\mathbb{R}$, $\varepsilon \in (0,1)$), we deduce that $\exp\left(
  -\frac{1}{2}(g_{\W^{**}}^{-1}(u))^2 \right)$ is upper bounded by
$$
\exp\left( -\frac{\varepsilon u^2}{2c^2(1+\varepsilon)}+
  \frac{\varepsilon (g_{\W^{**}}(0))^2}{2c^2} \right)\leq \exp\left(
  -\frac{\varepsilon u^2}{2c^2(1+\varepsilon)}+ \varepsilon
  w^2+\varepsilon \sup_{t \in [0,1]} (\W^{**}_t)^2\right)
$$
where we have used \eqref{E:majgWstar}. Combining this with
\eqref{E:double_enca} and taking the expectation with respect to
$\W^{**}$, we get that the quantity \eqref{E:intInOut} is smaller
than: $ c \int h_0(u) e^{-\frac{\varepsilon u^2}{2c(1+\varepsilon)}}
e^{\varepsilon w^2} E_{\W^{**}}\left( e^{ \varepsilon \sup_{t \in
      [0,1]} (\W^{**}_t)^2} \right).  $ The last expectation is finite
as soon as $\varepsilon$ is small enough, and thus
\eqref{E:borne_cond_maj} holds.
\end{proof}

\subsubsection{Proof of Proposition \ref{P:inversibleMalliaB}}
\label{S:Proof_of_Proposition}

To have shorter notations we set $\overline{\xr}^n_j=\int_0^1
\xr^n_{j+t} \dd \mu(t)$, for $j \ge 0$. First we prove:

\begin{lem}
\label{L:malliavin_markov}
Let us define $\Gamma^k$ the Malliavin covariance matrix of the vector
$\left( ( \overline{\xr}^n_j,\xr^n_{j+1})\right)_{j=0,\dots,k-1}$ of size $2k$.
Then this matrix is $a.s$ invertible and $E( \det(\Gamma^k)^{-p} ) \le
c(p,k).$
\end{lem}
\begin{proof}
  In the case $k=1$, the lemma reduces to Lemma
  \ref{L:mino_mallia_taille2}. For $k\ge 2$, we proceed by
  induction by establishing simple relations between the columns of $\Gamma^k$ (this simplification follows from the flow property of the process $\xr^n$).

  To see this, notice that firstly by \eqref{E:formulDX} if $t<k-1$
  and $s>k-1$, we have $D_t \xr^n_s=D_t \xr^n_{k-1} \yr^n_s
  (\yr^n_{k-1})^{-1}$; and secondly if $t>k-1$ and $s<k-1$, $D_t
  \xr^n_s=0$. Using these two properties, a calculation shows that if
  $(C_{j})_{j=1,\dots,2k}$ denote the columns of $\Gamma^k$, we have  the relation,
\begin{multline*}
  \left[C_{2k-1};C_{2k}\right]= \left[ \left( \int_{0}^{1}
      \yr^n_{k-1+s} (\yr^n_{k-1})^{-1} \dd \mu(s) \right) C_{2k-2};
    \yr^n_{k} (\yr^n_{k-1})^{-1}C_{2k-2}
  \right]\\+
  \left(\begin{array}{cc} 0 & 0 \\ \vdots& \vdots \\0 & 0 \\
      \multicolumn{2}{c} {\begin{bmatrix}\gamma_k\end{bmatrix}}
\end{array} \right)
\end{multline*}
where $\gamma_k$ is the matrix of size $2\times 2$ given by
$$
\begin{pmatrix}
  \int_{k-1}^{k} (D_t \overline{\xr}^n_{k-1})^2 \dd t & \int_{k-1}^{k}
  (D_t \xr^n_k) (D_t \overline{\xr}^n_{k-1}) \dd t
  \\
  \int_{k-1}^{k} (D_t \xr^n_k) (D_t \overline{\xr}^n_{k-1}) \dd t &
  \int_{k-1}^{k} (D_t \xr^n_k)^2 \dd t
\end{pmatrix}.$$
This proves that $\det{\Gamma^k}=\det{\Gamma^{k-1}} \det{\gamma_k}$.
But it can be seen that the matrix $\gamma_k$ has an expression
similar to $\gamma_{U^n,V^n}$ (but with integration interval shifted
from $[0,1]$ to $[k-1,k]$) from which we can prove $E(
(\det{\gamma_k})^{-p}) \le c(p)$.

The lemma then follows from induction on $k$.
%
\end{proof}
Now we can deduce the Proposition \ref{P:inversibleMalliaB}.
Recalling \eqref{E:defU0}--\eqref{E:defUk} we can find an invertible
matrix $M$ of size $2k \times 2k$ that maps
$\left((\overline{\xr}^n_j,\xr^n_{j+1})\right)_{j=0,\dots,k-1}$ into
a vector whose $k+1$ first components are exactly
$(U_0^n,\dots,U_{k}^n)$. Denoting $\hat{\Gamma}^k$ the Malliavin
covariance matrix of the image by $M$ of
$\left((\overline{\xr}^n_j,\xr^n_{j+1})\right)_{j=0,\dots,k-1}$, we
have $\hat{\Gamma}^k=M \Gamma^k M^\star$. Thus, Lemma
\ref{L:malliavin_markov} yields $E( (\det(\hat{\Gamma}^k))^{-p} )
\le c(p,k)$ since $M$ is invertible. Observing that the Malliavin
covariance matrix $K(\theta)$ is the matrix extracted from the $k+1$
first rows and columns of $\hat{\Gamma}^k$ we deduce Proposition
\ref{P:inversibleMalliaB}.


\subsection{Some estimates on the change of measures}
For this section we denote by $\xr$ the canonical process on
$\mathcal{C}([0,\infty))$ and we consider the random
variable on this space defined by $H=f(U_0,\dots,U_k)$, where
$(U_0,\dots,U_k)$ is given by \eqref{E:defU0}--\eqref{E:defUk} with the canonical
process $\xr$ in place of $\xr^{\theta,n}$;
we denote by $E_{\theta,x_0}^n$ the expectation with respect to the measure induced
on the canonical space by the law of
$\xr^{\theta,n}$ solution of \eqref{E:defXronde}.

\begin{lem} \label{L:stabilite_taille}
There exist $r\ge1$ and a constant $c(k)\ge 0$, such that  $\forall H=f(U_0,\dots,U_k)\ge0$, $\forall \theta,\theta'
\in \Theta$, $\forall x_0 \in \mathbb{R}$, we have
$$
E_{\theta',x_0}^n[H] \le c(k) E_{\theta,x_0}^n[H^{r}]^{\frac{1}{r}}.
$$
\end{lem}
\begin{proof}
Recalling the notation of Section \ref{S:Existence_of}
we denote $\pUV_{x_0}^n(u,v,\theta)$ the density of the vector
\eqref{E:coupleUV} and for
$j=0,\dots,k-1$ we let
\begin{equation}\label{E:defZj}
Z_{j,\theta,\theta'}=
\frac{
\pUV_{\xr_{j}}^n
\left(  \int_{0}^1
(\xr_{j+s}-\xr_{j}) \dd \mu (s) ,
(\xr_{j+1}-\xr_{j}),
\theta'\right)
}
{\pUV_{\xr_j}^n
\left(
\int_{0}^{1}
( \xr_{j+s}-\xr_j )\dd \mu (s),
(\xr_{j+1}-\xr_{j}),\theta \right)
}.
\end{equation}
Then using the Markov property of the process $\xr$ under the laws
$P^{n}_{\theta}$ and $P^{n}_{\theta'}$, we have
\begin{align*}
E^n_{\theta',x_0}[H]&=E^n_{\theta,x_0}\left[
H \prod_{j=0}^{k-1} Z_{j,\theta,\theta'}\right]
\le E^n_{\theta,x_0}[H^r]^{\frac{1}{r}}
E^n_{\theta,x_0} \left[ \prod_{j=0}^{k-1} (Z_{j,\theta,\theta'})^{r'} \right]^{\frac{1}{r'}},
\end{align*}
where $r$ and $r'$ are conjugate exponents. But we know by Theorem
\ref{T:densiteUV} that there exist two constants $0<c_2\le c_1$
(uniform w.r.t. $\theta,x_0,n$) such that
$$
c_1^{-1}e^{-c_1(u^2+v^2)} \le \pUV^n_{x_0}(u,v,\theta) \le
c_2^{-1}e^{-c_2(u^2+v^2)}.
$$
Then one can bound the conditional expectation
$E_{\theta,x_0}^n[(Z_{k-1,\theta,\theta'})^{r'} \mid \xr_s, s\le (k-1) ]$ by
$$
\frac{c_1^{(r'-1)}}{c_2^{r'}}
\int_{\mathbb{R}^2} e^{(u^2+v^2)(-r'c_2+(r'-1)c_1 )} \dd u \dd v .
$$
But if $r$ is chosen large enough such that $r'$ is sufficiently
close to 1 the latter integral converges and is equal to some constant $\kappa$.
Proceeding by induction we get:
\begin{equation*}
E_{\theta,x_0}^n \left[ \prod_{j=0}^{k-1} (Z_{j,\theta,\theta'})^{r'} \right]^{\frac{1}{r'}}
\le E_{\theta,x_0}^n \left[ \prod_{j=0}^{k-2} (Z_{j,\theta,\theta'})^{r'} \right]^{\frac{1}{r'}} \kappa^{\frac{1}{r'}}
\le \dots
\le \kappa^{k/r'}
\end{equation*}
 which gives the result.
\end{proof}
\begin{lem}\label{L:stabilite_centrage}
There exist $c(k)\ge 0$ and $\alpha\ge 1$ such that
$\forall H=f(U_0,\dots,U_k)$ (with $E_{\theta,x_0}^n|H|^{\alpha}<+\infty$), $\forall \theta,\theta'
\in \Theta$, $\forall x_0 \in \mathbb{R}$, we have
\begin{equation}\label{E:stabilite_centrage_but}
\abs{E_{\theta',x_0}^n [H]- E_{\theta,x_0}^n[H]} \le c(k) \abs{\theta-\theta'}
 [E_{\theta,x_0}^n|H|^{\alpha}]^{\frac{1}{\alpha}}.
\end{equation}
\end{lem}
\begin{proof}
Using the notations of Lemma \ref{L:stabilite_taille}, we write
\begin{align*}
E_{\theta',x_0}^n[H]-E_{\theta,x_0}^n[H]
&=E_{\theta,x_0}^n\left[ \left( \prod_{j=0}^{k-1} Z_{j,\theta,\theta'} -1\right) H \right]
\\
&=\sum_{i=0}^{k-1} E_{\theta,x_0}^n \left[
\left( Z_{i,\theta,\theta'} -1 \right) \prod_{j=i+1}^{k-1} Z_{j,\theta,\theta'}
 H \right].
\end{align*}
Thus for conjugate exponents $\alpha$ and $\beta$, the left
hand side of \eqref{E:stabilite_centrage_but} is bounded by
\begin{align*}
&
\sum_{i=0}^{k-1} \left[E_{\theta,x_0}^n  \abs{H}^\alpha\right]^{\frac{1}{\alpha}}
E_{\theta,x_0}^n \left[\abs{ Z_{i,\theta,\theta'} -1 }^\beta
\prod_{j=i+1}^{k-1} (Z_{j,\theta,\theta'})^\beta
\right]^{\frac{1}{\beta}}
\\
=&
\sum_{i=0}^{k-1} \left[E_{\theta,x_0}^n  \abs{H}^\alpha\right]^{\frac{1}{\alpha}}
E_{\theta,x_0}^n \left[\abs{ Z_{i,\theta,\theta'} -1 }^\beta
E_{\theta,x_0}^n \left[  \prod_{j=i+1}^{k-1} (Z_{j,\theta,\theta'})^\beta
\mid \xr_s, s \le i+1 \right]
\right]^{\frac{1}{\beta}}
\end{align*}
Using the Markov property of $\xr$ it can be shown exactly as in Lemma
\ref{L:stabilite_taille} that the conditional expectation in the equation above
is finite, as soon as $\beta$ is small enough and bounded by $\kappa^{k-i-1}$.
Thus by Lemma \ref{L:Zmoins_un} below, we deduce
$\abs{E_{\theta',x_0}[H]- E_{\theta,x_0}[H]}  \le c(\beta) \abs{\theta-\theta'}
\left[E_{\theta,x_0}^n
 \abs{H}^\alpha\right]^{\frac{1}{\alpha}} \sum_{i=0}^{k-1}
 {\kappa}^{\frac{k-i-1}{\beta}}$.
\end{proof}
\begin{lem}\label{L:Zmoins_un}
There exists $\overline{\beta}>1$ such that for all $1<\beta \le \overline{\beta}$
we have:
$$
E_{\theta,x_0}^n \left[\abs{ Z_{i,\theta,\theta'} -1 }^\beta \right]^{\frac{1}{\beta}}
\le c(\beta)\abs{\theta-\theta'}.
$$
\end{lem}
\begin{proof}
Using the expression of $Z_{i,\theta,\theta'}$, and the equation
\eqref{E:coupleUV} with the Markov property, it suffices to bound
the quantity
\begin{equation}\label{E:Zmoins_un_but}
E_{x_0,\theta}^n
\left[ \abs{
\frac{\pUV_{x_0}^n \left( U^n,V^n,\theta' \right)-\pUV_{x_0}^n \left(U^n,V^n,\theta \right)}
{\pUV_{x_0}^n \left(U^n,V^n,\theta \right)}
}^{\beta}
\right]^{\frac{1}{\beta}}.
\end{equation}
By Theorem \ref{T:likehood_block} with $k=1$ the function $\theta
\to \pUV_{x_0}(U^n,V^n,\theta)$ is absolutely continuous and we can
write the quantity above as: $E_{x_0,\theta}^n \left[ \abs{ \frac{
\int_{\theta}^{\theta'}
 \dot{\pUV}_{x_0}^n \left( U^n,V^n,s \right) \dd s}
{\pUV_{x_0}^n \left(U^n,V^n,\theta \right)}
}^{\beta}
\right]^{\frac{1}{\beta}}$.
Using first the Minkowski inequality, a change of measure and
then the H{\"o}lder inequality one finds the following bounds for this quantity:
\begin{multline*}
\int_{\theta}^{\theta'}
E_{x_0,s}^n
\left[
\abs{
 \frac{ \dot{\pUV}^n_{x_0} \left( U^n,V^n,s \right) }
{\pUV_{x_0}^n \left(U^n,V^n,\theta \right)}}^{\beta}
\frac{  \pUV_{x_0}^n \left( U^n,V^n,\theta \right) }
{\pUV_{x_0}^n \left(U^n,V^n,s\right)}
\right]^{\frac{1}{\beta}} \dd s
\\
\le
 \int_{\theta}^{\theta'}
E_{x_0,s}
\left[
\abs{
\frac{
 \dot{\pUV}_{x_0}^n \left( U^n,V^n,s\right)}
{\pUV_{x_0}^n \left(U^n,V^n,s \right)}}^{\beta \alpha'}
\right]^{\frac{1}{\beta\alpha'}}
E_{x_0,s}
\left[
\abs{
\frac{
 \pUV_{x_0}^n \left( U^n,V^n,s \right)}
{\pUV_{x_0}^n \left(U^n,V^n,\theta \right)}}^{(\beta-1) \beta'}
\right]^{\frac{1}{\beta\beta'}}
\dd s
\end{multline*}
for two conjugate exponents $\alpha'$ and $\beta'$. But the first
expectation in the right hand side above is bounded by Corollary
\ref{C:majdotpp} (with $k=1$) for all choices of $\alpha'$, $\beta$.
The second expectation can be bounded if $(\beta-1)\beta'$ is close
enough to zero by using \eqref{E:maj_min_densite} as in the proof of
Lemma \ref{L:stabilite_taille}. This gives that
\eqref{E:Zmoins_un_but} is smaller than $c \abs{\theta-\theta'}$.
\end{proof}
\subsection{A technical lemma}
\begin{lem}\label{L:diff_chi_deux}
Let $(G_0,\dots,G_k)$ be a centered Gaussian vector with invertible covariance
matrix $C_{k+1}$ and let us denote by $C_{k-1}$ the covariance matrix of
$(G_1,\dots,G_{k-1})$. Then,
\begin{equation}\label{E:chi_deux_deux}
\sum_{0\le j,j' \le k}G_j[C_{k+1}]^{-1}_{j,j'} G_{j'} -
\sum_{1\le j,j' \le k-1}G_j[C_{k-1}]^{-1}_{j,j'} G_{j'},
\end{equation}
is a $\chi^2(2)$ random variable.
\end{lem}
\begin{proof} Write the Gram-Schmidt orthonormalization procedure for
the $\mathbf{L^2}$ vectors $G_1,\dots,G_k,G_0$ as:
$$
\begin{bmatrix}
\mathcal{H}_0\\ \vdots \\ \mathcal{H}_k
\end{bmatrix}
=
P_{k}
\begin{bmatrix}
G_0\\ \vdots \\ G_k
\end{bmatrix},
$$
where the variables $\mathcal{H}_0, \dots , \mathcal{H}_k $ are i.i.d. with
standard Gaussian law and $P_k$ is some triangular matrix.
Then a few linear algebra
shows that \eqref{E:chi_deux_deux} is equal to
$\sum_{j=0}^k \mathcal{H}^2_j -\sum_{j=1}^{k-1} \mathcal{H}^2_j=
\mathcal{H}^2_0+\mathcal{H}^2_{k}$
and thus is chi-square distributed.
\end{proof}

\end{document}